\documentclass[reqno,12pt]{amsart}
\usepackage[active]{srcltx}
\usepackage {color}
\usepackage{amsmath,amsthm,amssymb}
\usepackage{epsfig}
\usepackage{graphicx}
\usepackage{amssymb}
\usepackage{latexsym}
\usepackage{pdfpages}

\usepackage{enumitem}

\usepackage{ulem} 

\usepackage{ifpdf}

\ifpdf 
\usepackage[hidelinks]{hyperref}
\else 
\fi 

\usepackage[usenames,dvipsnames]{pstricks}
\usepackage{epsfig}
\usepackage{pst-grad} 
\usepackage{pst-plot} 

\usepackage{color}

\newtheorem{theorem}{Theorem}[section]
\newtheorem{lemma}[theorem]{Lemma}
\newtheorem{definition}[theorem]{Definition}
\newtheorem{proposition}[theorem]{Proposition}
\newtheorem{corollary}[theorem]{Corollary}
\newtheorem{remark}[theorem]{Remark}

\newtheorem*{theorem*}{\it Theorem}

\def\vint_#1{\mathchoice%
	{\mathop{\kern 0.2em\vrule width 0.6em height 0.69678ex depth -0.58065ex
			\kern -0.8em \intop}\nolimits_{\kern -0.4em#1}}%
	{\mathop{\kern 0.1em\vrule width 0.5em height 0.69678ex depth -0.60387ex
			\kern -0.6em \intop}\nolimits_{#1}}%
	{\mathop{\kern 0.1em\vrule width 0.5em height 0.69678ex
			depth -0.60387ex
			\kern -0.6em \intop}\nolimits_{#1}}%
	{\mathop{\kern 0.1em\vrule width 0.5em height 0.69678ex depth -0.60387ex
			\kern -0.6em \intop}\nolimits_{#1}}}
\def\vintslides_#1{\mathchoice%
	{\mathop{\kern 0.1em\vrule width 0.5em height 0.697ex depth -0.581ex
			\kern -0.6em \intop}\nolimits_{\kern -0.4em#1}}%
	{\mathop{\kern 0.1em\vrule width 0.3em height 0.697ex depth -0.604ex
			\kern -0.4em \intop}\nolimits_{#1}}%
	{\mathop{\kern 0.1em\vrule width 0.3em height 0.697ex depth -0.604ex
			\kern -0.4em \intop}\nolimits_{#1}}%
	{\mathop{\kern 0.1em\vrule width 0.3em height 0.697ex depth -0.604ex
			\kern -0.4em \intop}\nolimits_{#1}}}

\def\R{\mathbb R}
\def\N{\mathbb N}

\numberwithin{equation}{section}

\def\1{\raisebox{2pt}{\rm{$\chi$}}}

\def\Xint#1{\mathchoice
	{\XXint\displaystyle\textstyle{#1}}%
	{\XXint\textstyle\scriptstyle{#1}}%
	{\XXint\scriptstyle\scriptscriptstyle{#1}}%
	{\XXint\scriptscriptstyle\scriptscriptstyle{#1}}%
	\!\int}
\def\XXint#1#2#3{{\setbox0=\hbox{$#1{#2#3}{\int}$}
		\vcenter{\hbox{$#2#3$}}\kern-.5\wd0}}

\def\dashint{\Xint-}

\def\Xint#1{\mathchoice
	{\XXint\displaystyle\textstyle{#1}}%
	{\XXint\textstyle\scriptstyle{#1}}%
	{\XXint\scriptstyle\scriptscriptstyle{#1}}%
	{\XXint\scriptscriptstyle\scriptscriptstyle{#1}}%
	\!\int}
\def\XXint#1#2#3{{\setbox0=\hbox{$#1{#2#3}{\int}$}
		\vcenter{\hbox{$#2#3$}}\kern-.5\wd0}}

\def\dashint{\Xint-}

\newcommand{\twopartdef}[4]
{
	\left\{
	\begin{array}{ll}
		#1 & #2 \\
		#3 & #4
	\end{array}
	\right.
}

\newcommand{\threepartdef}[6]
{
	\left\{
	\begin{array}{lll}
		#1 & \mbox{if } #2 \\
		#3 & \mbox{if } #4 \\
		#5 & \mbox{if } #6
	\end{array}
	\right.
}

\definecolor{violet(ryb)}{rgb}{0.53, 0.0, 0.69}
\begin{document}
	
\title[$p$-Laplacian gradient flow in metric measure spaces]{ On the $p$-Laplacian evolution equation in metric measure spaces}
	
\author[W. G\'{o}rny and J. M. Maz\'on]{Wojciech G\'{o}rny and Jos\'e M. Maz\'on}
	
\address{ W. G\'{o}rny: Faculty of Mathematics, Universit\"at Wien, Oskar-Morgerstern-Platz 1, 1090 Vienna, Austria; Faculty of Mathematics, Informatics and Mechanics, University of Warsaw, Banacha 2, 02-097 Warsaw, Poland
\hfill\break\indent
{\tt  wojciech.gorny@univie.ac.at }
}
	
\address{J. M. Maz\'{o}n: Departamento de An\`{a}lisis Matem\`atico,
Universitat de Val\`encia, Dr. Moliner 50, 46100 Burjassot, Spain.
\hfill\break\indent
{\tt mazon@uv.es }}
	
%
%
	
	
\setcounter{tocdepth}{1}

\date{\today}

\begin{abstract}
The  $p$-Laplacian evolution equation in metric measure spaces has been studied as the gradient flow in $L^2$ of the $p$-Cheeger energy (for $1 < p < \infty$). In this paper,  using the first-order differential structure on a metric measure space introduced by Gigli, we characterize the subdifferential in $L^2$ of the $p$-Cheeger energy. This gives rise to a new definition of the $p$-Laplacian operator in metric measure spaces, which allows us to work with this operator in more detail. In this way, we introduce a new notion of solutions to the $p$-Laplacian evolution equation in metric measure spaces.  For $p = 1$, we obtain a Green-Gauss formula similar to the one by Anzellotti for Euclidean spaces, and use it to characterise the $1$-Laplacian operator and study the total variation flow. We also study the asymptotic behaviour of the solutions of the $p$-Laplacian evolution equation, showing that for $1 \leq p < 2$ we have finite extinction time.
\end{abstract}
	
\maketitle
	
%
%
%
%
%
%

{\renewcommand\contentsname{Contents }
\setcounter{tocdepth}{3}
\tableofcontents}

\section{Introduction}

The gradient flows can be regarded as the paradigm of dissipative evolution and have hence attracted a constant attention during the last four decades starting
from the fundamental work by K\"omura \cite{K}, Crandall-Pazy \cite{CP}, and Brezis \cite{Brezis0} about gradient flows of convex functionals in Hilbert spaces. This gives rise to a notion of solutions to evolution equations which are gradient flows of said functional; furthermore, solutions constructed in this way have regularising properties. For an exhaustive treatment of this subject, see \cite{Brezis}. More recently, following the pioneering work by Otto \cite{Otto}, an even larger class of PDE problems have been translated into gradient flows by resorting to probability spaces endowed with the Wasserstein metric (see also the monograph by Ambrosio, Gigli and Savar\'e \cite{AGSBook}).

The study of gradient flows in metric spaces faces some additional difficulties (for a standard reference, see \cite{AGSBook}). The standard problem of this type is the heat flow. Using the semigroup approach, it has been studied by Ambrosio, Gigli and Savar\'{e} in \cite{AGS2}: in a metric measure space $(\mathbb{X},d,\nu)$, the authors define it as the gradient flow in $L^2(\mathbb{X},\nu)$ of the Dirichlet-Cheeger energy and study its properties under the assumption that $(\mathbb{X},d,\nu)$ has Ricci curvature bounded from below. Then, this gradient flow coincides with gradient flow of the Bolztmann entropy with respect to the Wasserstein distance in the space of probabilities. The same authors in \cite{AGS1} (see also \cite{Kell}) studied the case of the $p$-Heat flow as the gradient flow in $L^2(\mathbb{X},\nu)$ of the $p$-Cheeger energy and Kell in \cite{Kell} showed that the $p$-Heat flow is the gradient flow of the Renyi entropy in the $q$-Wasserstein space, where $\frac{1}{p} + \frac{1}{q} = 1$.

In these papers the gradient flow in $L^2(\mathbb{X},\nu)$ is defined in the framework of maximal monotone operator in Hilbert spaces and the corresponding $p$-Laplacian operator is defined through the subdifferential of the $p$-Cheeger energy, but without giving a characterisation of it. For this reason (but also to deal with some non-convex functionals), several variational methods have been employed to characterise the solutions, the main ones being the minimising movements approach (see \cite{AGSBook}), evolution variational inequalities (also see \cite{AGSBook}) and the weighted energy-dissipation principle (see \cite{MO},\cite{RSSS}).

Our aim is to  study the gradient flows which arise from convex functionals in metric measure spaces. We introduce a new framework based on the first-order linear differential structure introduced by Gigli in \cite{Gig}. In this paper, we present it in the case of the $p$-Laplacian evolution equation for $1 < p < \infty$, i.e. the gradient flow of the $p$-Cheeger energy, and the total variation flow (for $p = 1$). We characterise the subdifferentials in $L^2(\mathbb{X},\nu)$ of the $p$-Cheeger energy and the total variation functional using this linear structure, and then apply this characterisation to introduce a new notion of solutions to the corresponding evolution problems. Furthermore, this allows us to study these operators more directly, and in particular we prove that they are completely accretive. For the sake of the presentation, we restrict ourselves to these cases; however, as we can see in the paper, this approach is quite flexible and can be possibly applied also to other settings, such as  more general convex functionals, sufficiently regular bounded domains or less regular initial data.

Let us shortly describe the contents of the paper. In Section \ref{sec:preliminaries}, we recall all the notions about analysis on metric measure spaces required in this paper. We will work under the standard assumptions that the metric space $(\mathbb{X},d)$ is complete and separable. Furthermore, we will require that $\nu$ is a nonnegative Radon measure which is finite on bounded sets. { In Section \ref{subsec:subdifferential}, we recall the notions of subdifferential in convex analysis and solutions to abstract Cauchy problems in Hilbert spaces.} In Section \ref{subsec:sobolevbv}, we recall the definitions of the Sobolev spaces $W^{1,p}(\mathbb{X},d,\nu)$ and the space of functions of bounded variation $BV(\mathbb{X},d,\nu)$ in the metric setting. Then, in Section \ref{subsec:giglistructure} we recall the construction of the first-order differential structure on the metric measure space $(\mathbb{X},d,\nu)$ introduced by Gigli (see \cite{Gig}) using the machinery of $L^p(\nu)$-normed modules.

{ In Section \ref{sec:plaplace}, we study the first example of our proposed definition of solutions to the gradient flow of a convex functional: the $p$-Laplacian evolution equation (for $1 < p < \infty$). It is understood as the gradient flow of the $p$-Cheeger energy $\mathsf{Ch}_p: L^2(\mathbb{X},\nu) \rightarrow [0,+\infty]$ defined by the formula
\begin{equation*}
\mathsf{Ch}_p(u) = \twopartdef{\displaystyle\frac{1}{p} \int_{\mathbb{X}} |Du|^p \, d\nu}{u \in W^{1,p}(\mathbb{X},d,\nu)}{+ \infty}{u \in L^2(\mathbb{X},\nu) \setminus W^{1,p}(\mathbb{X},d,\nu).}
\end{equation*}
The $p$-Cheeger energy is defined on $L^2(\mathbb{X},\nu)$ in order to use the classical Hilbertian theory of gradient flows. In order to provide a new characterisation of solutions to the gradient flow associated to $\mathsf{Ch}_p$, we study the subdifferential of $\mathsf{Ch}_p$ (and the associated $p$-Laplacian operator) and express it in terms of the differential structure introduced by Gigli in \cite{Gig}. To achieve this goal, we use the techniques of convex duality, in particular the Fenchel-Rockafellar duality theorem. Then, we use this characterisation together with the classical theory of maximal monotone operators to get existence and uniqueness of solutions to the associated gradient flow, we show that the subdifferential is completely accretive (which implies a contraction property for the gradient flow), provide a set of equivalent notions of solutions and discuss their relationship with the definitions typically encountered in the literature. Let us note that our results in this Section do not require doubling, Poincar\'e, or curvature assumptions on the metric measure space $(\mathbb{X},d,\nu)$.}

In Section \ref{sec:Anzellotti}, we make a preparation to study the total variation flow using this technique. We assume additionally that $\nu$ is doubling and that it satisfied a $(1,1)$-Poincar\'e inequality. The main goal of this Section is to introduce a metric analogue of the Anzellotti pairing between a function in $BV(\mathbb{X},d,\nu)$ and a vector field with integrable divergence, which would serve as a replacement for the differential structure due to Gigli for functions in $BV(\mathbb{X},d,\nu)$. Then, we prove that the constructed pairing satisfies a Gauss-Green formula and a co-area formula.

In Section \ref{sec:TVflow}, we study the total variation flow, which is the gradient flow of the total variation functional $\mathcal{TV}: L^2(\mathbb{X},\nu) \rightarrow [0,+\infty]$ defined by the formula
\begin{equation*}
\mathcal{TV}(u) = \twopartdef{\displaystyle \int_{\mathbb{X}} |Du|_\nu}{u \in BV(\mathbb{X},d,\nu)}{+ \infty}{u \in L^2(\mathbb{X},\nu) \setminus BV(\mathbb{X},d,\nu).}
\end{equation*}
Again, it is defined on $L^2(\mathbb{X},\nu)$ so that we can use the classical Hilbertian theory of gradient flows. We use the Gauss-Green formula developed in the previous Section to characterise the subdifferential of the total variation functional and then study the total variation flow in the metric setting. We also use the co-area formula developed in the previous Section to we show that the subdifferential is completely accretive, so that the associated gradient flow has a contraction property.

Section \ref{sec:asymptotics} is devoted to the asymptotics of the flow of the $p$-Cheeger energy and the total variation flow. { We show that the general theory developed in \cite{BB} for convex, $p$-homogenous and coercive functionals in Hilbert spaces can be applied in the situation considered in the present paper: assuming that a version of the Poincar\'e inequality or the Sobolev inequality holds, which implies coercivity of the functionals $\mathsf{Ch}_p$ and $\mathcal{TV}$, we show that a direct application of the results from \cite{BB} implies existence of a finite extinction time for $1 \leq p < 2$ and some upper bounds for the solution.}

Finally, in Section \ref{sec:particularcases} we present some important special cases in which our characterisation of solutions is more detailed. To be exact, we apply our results in two situations when there is a known characterisation of the abstract differential structure introduced by Gigli in terms of standard vector fields and $1$-forms in coordinates (see \cite{LPR,LP}). The first one is the Euclidean space $\mathbb{R}^N$ equipped with a nonnegative Radon measure (called the weighted Euclidean space) and the other one is a reversible Finsler manifold, also equipped with a nonnegative Radon measure. Let us note that the except for a few special choices of weights and exponents, the results in this paper are new also in these two particular cases.

\section{Preliminaries}\label{sec:preliminaries}

{
\subsection{Convex functions and subdifferentials}\label{subsec:subdifferential}

Let $H$ be a real Hilbert space equipped with a scalar product $\langle \cdot, \cdot \rangle_H$ and the norm $\Vert u \Vert_H = \sqrt{\langle u, u \rangle_H}$ derived from the scalar product. Given a functional $\mathcal{F} : H \to (-\infty, \infty]$, we call the set
$$D(\mathcal{F}) : = \{ u \in H \ : \ \mathcal{F}(u) < + \infty \}$$
the {\it effective domain} of $\mathcal{F}$. The functional $\mathcal{F}$ is said to be proper if $D(\mathcal{F})$ is non-empty.
Furthermore, we say that $\mathcal{F}$ is {\it lower semi-continuous} if for every $c \in \R$ the sublevel set
$$E_c := \{ u \in D(\mathcal{F}) \ : \ \mathcal{F}(u) \leq c \}$$
is closed in $H$.

Given a proper and convex functional $\mathcal{F} : H \to (-\infty, \infty]$, its {\it subdifferential} is the set
\begin{equation}\label{subdif2}\partial \mathcal{F} := \left\{(u,h) \in H \times H \ : \ \mathcal{F}(u+v) - \mathcal{F}(u) \geq \langle h, v \rangle_H \ \ \forall \, v \in H \right\}.
\end{equation}
It is a generalisation of the derivative; in the case when $\mathcal{F}$ is Fr\'{e}chet differentiable, its subdifferential is single-valued and equals the Fr\'{e}chet derivative.

Now, consider a multi-valued operator $A$ on $H$, i.e. a mapping $A: H \rightarrow 2^H$. It is standard to identify $A$ with its graph in the following way: for every $u \in H$, we set
$$ Au : = \left\{ v \in H \ : \  (u,v) \in A \right\}.$$
We denote the domain of $A$ by
$$D (A) : = \left\{ u \in H \ : \ Au \not= \emptyset \right\}$$
and the range of $A$ by
$$R(A) : = \bigcup_{u \in D(A)} Au.$$
An multi-valued operator $A$ on $H$ is {\it monotone} if
$$\langle u - \hat{u}, v - \hat{v} \rangle_H \geq 0 \quad \hbox{for all} \ (u,v), (\hat{u}, \hat{v}) \in A.$$
If there is no monotone operator which strictly contains $A$, we say that $A$ is maximal monotone. A classical example of a multi-valued operator is the subdifferential; if $\mathcal{F} : H \to (-\infty, \infty]$ is convex and lower semi-continuous, then $\partial \mathcal{F}$ is a maximal monotone multi-valued operator on $H$.

For $1 \leq p < \infty$, we denote
$$L^p(a,b; H):= \left\{ u : [a,b] \rightarrow H \ \hbox{measurable such that} \  \int_a^b \Vert u(t) \Vert_H^p dt < \infty \right\}$$
and
$$W^{1,p}(a,b;H):= \left\{ u \in L^p(a,b; H) \ \hbox{and} \ \exists \, v \in L^p(a,b; H) \ : \ u(t) - u(a) = \int_a^t v(s) ds \ \ \forall \, t \in (a,b)  \right\}.$$
By \cite[Corollaire A.2]{Brezis}, if $u \in W^{1,p}(a,b;H)$, it is differentiable for almost all $t \in (a,b)$ and
$$u(t) - u(a) = \int_a^t \frac{du}{dt}(s) \, ds \ \ \forall \, t \in (a,b).$$
We also set $W_{loc}^{1,p}(0,T;H)$ to be the space of all functions $u$ with the following property: for all $0 < a < b < T$, we have that $u \in W^{1,p}(a,b;H)$.

Consider the abstract Cauchy problem
\begin{equation}\label{ACP1}
\left\{ \begin{array}{ll} \frac{du}{dt} + \partial \mathcal{F} (u(t)) \ni 0 \, \quad &t \in (0, T),
  \\[10pt] u(0) = u_0, \quad &u_0 \in H. \end{array} \right.
\end{equation}

\begin{definition}\label{StronSol} We say that $u \in C([0,T]; H)$ is a {\it strong solution} of problem \eqref{ACP1}, if the following conditions hold: $u \in W_{loc}^{1,2}(0,T;H)$; for almost all $t \in (0,T)$ we have $u(t) \in D(\partial \mathcal{F})$; and it satisfies \eqref{ACP1}.
\end{definition}

We are now in position to state the celebrated {\it Brezis-Komura Theorem} (see \cite{Brezis}).

\begin{theorem}\label{BKTheorem} Let $\mathcal{F} : H \to (-\infty, \infty]$ be a proper, convex, and lower semi-continuous functional. Given $u_0 \in \overline{D(\mathcal{F})}$, there exists a unique strong solution of the abstract Cauchy problem \eqref{ACP1}. Moreover, we have that $\sqrt{t}\frac{du}{dt} \in L^2(0, T; H)$, and $u \in W^{1,2}(0, T; H)$ whenever $u_0 \in D(\mathcal{F})$.

If we denote by $S(t) u_0$ the unique strong solution $u(t)$ of the abstract Cauchy problem \eqref{ACP1}, then $S(t) : \overline{D(\mathcal{F})} \rightarrow H$ is a continuous semigroup satisfying the $T$-contraction property
$$\Vert (S(t) u_0 - S(t) v_0)^{\pm} \Vert \leq \Vert (u_0 -  v_0)^{\pm} \Vert \quad \forall \, t >0$$
for all $u_0, v_0 \in \overline{D(\mathcal{F})}$.
\end{theorem}
 }

\subsection{Sobolev and BV functions in metric measure spaces}\label{subsec:sobolevbv}
	
Let $(\mathbb{X}, d, \nu)$ be a metric measure space. Given $p \in [1,\infty)$, there are a number of possible definitions of Sobolev spaces on $\mathbb{X}$, most prominently via $p$-upper gradients, $p$-relaxed slopes, and via test plans. On complete and separable metric spaces equipped with a  nonnegative Borel measure finite on bounded sets, all these definitions agree (see \cite{AGS1,DiMarinoTh}); since in this paper we will work under these assumptions, we will choose the most suitable definition for our purposes: the Newtonian spaces. We follow the presentation in \cite{BjBj}.

\begin{definition}{\rm We say that a measure $\nu$ on a metric space $\mathbb{X}$ is {\it doubling}, if there exists a constant $C_d \geq 1$ such that following condition holds:
\begin{equation}
0 < \nu(B(x,2r)) \leq C_d \, \nu(B(x, r)) < \infty
\end{equation}
for all $x \in \mathbb{X}$ and $r > 0$. The constant $C_d$ is called the doubling constant of $\mathbb{X}$. }
\end{definition}

\begin{definition}{\rm  We say that $\mathbb{X}$ supports a {\it weak $(1,p)$-Poincar\'{e} inequality} if there exist constants $C_P > 0$ and $\lambda \geq 1$ such that for all balls $B \subset \mathbb{X}$, all measurable functions $f$ on $\mathbb{X}$ and all upper gradients $g$ of $f$,
$$\dashint_B \vert f - f_B \vert d \nu \leq C_P r \left( \dashint_{\lambda B} g^p d\nu \right)^{\frac1p},  $$
where $r$ is the radius of $B$ and
$$f_B:= \dashint_B f d\nu := \frac{1}{\nu(B)} \int_B f d\nu.$$
The space $\mathbb{X}$ supports a {\it $(1,p)$-Poincar\'{e} inequality} if it supports a weak $(1,p)$-Poincar\'{e} inequality with $\lambda = 1$.}
\end{definition}

By H\"older's inequality, it is easy to see that if	$\mathbb{X}$ supports a  weak $(1,p)$-Poincar\'{e} inequality, then it supports a weak $(1,q)$-Poincar\'{e} inequality for every $q \geq p$.

We say that a Borel function $g$ is an {\it upper gradient} of a Borel function $u: \mathbb{X} \rightarrow \R$ if for all curves $\gamma: [0,l_\gamma] \rightarrow \mathbb{X}$ we have
$$ \left\vert u(\gamma(l_\gamma)) - u(\gamma(0)) \right\vert \leq \int_\gamma g:= \int_0^{l_\gamma} g(\gamma(t)) \vert \dot{\gamma}(t) \vert dt \,  ds,$$
where $$\vert \dot{\gamma}(t) \vert:= \lim_{\tau \to 0} \frac{\gamma(t + \tau) - \gamma(t)}{\tau}$$
is the {\it metric speed} of $\gamma$.
If this inequality holds for $p$-almost every curve, i.e. the $p$-modulus (see for instance \cite[Definition 1.33]{BjBj}) of the family of all curves for which it fails equals zero, then we say that $g$ is a {\it $p$-weak upper gradient} of $u$.

The Sobolev-Dirichlet class $D^{1,p}(\mathbb{X})$ consists of all Borel functions $u: \mathbb{X} \rightarrow \R$ for which there exists  an upper gradient (equivalently: a $p$-weak upper gradient) which lies in $L^p(\mathbb{X},\nu)$. The Sobolev space $W^{1,p}(\mathbb{X}, d, \nu)$ is defined as
$$W^{1,p}(\mathbb{X}, d, \nu):= D^{1,p}(\mathbb{X}) \cap L^p(\mathbb{X}, \nu).$$
In the literature, this space is sometimes called the Newton-Sobolev space (or Newtonian space) and is denoted $N^{1,p}(\mathbb{X})$. The space $W^{1,p}(\mathbb{X},d,\nu)$ is endowed with the norm
\begin{equation*}
\| u \|_{W^{1,p}(\mathbb{X},d,\nu)} = \bigg( \int_{\mathbb{X}} |u|^p \, d\nu + \inf_g \int_{\mathbb{X}} g^p \, d\nu \bigg)^{1/p},
\end{equation*}
where the infimum is taken over all upper gradients of $u$. Equivalently, we may take the minimum over the set of all $p$-weak upper gradients, see \cite[Lemma 1.46]{BjBj}. Under the assumptions that $\nu$ is doubling and a weak $(1,p)$-Poincar\'e inequality is satisfied, Lipschitz functions are dense in $W^{1,p}(\mathbb{X},d,\nu)$ (see \cite[Theorem 5.1]{BjBj}). Let us also stress that the same definition may be applied to open subsets $\Omega \subset \mathbb{X}$.

For every $u \in W^{1,p}(\mathbb{X},d,\nu)$ (even $u \in D^{1,p}(\mathbb{X})$), there exists a minimal $p$-weak upper gradient $|Du| \in L^p(\mathbb{X},\nu)$, i.e. we have
\begin{equation*}
|Du| \leq g \quad \nu-\mbox{a.e.}
\end{equation*}
for all $p$-weak upper gradients $g \in L^p(\mathbb{X},\nu)$ (see \cite[Theorem 2.5]{BjBj}). It is unique up to a set of measure zero. In particular, we may simply plug in $|Du|$ in the infimum in the definition of the norm in $W^{1,p}(\mathbb{X},d,\nu)$. Moreover, in \cite{AGS1} (see also \cite{DiMarinoTh}) it was proved that on complete and separable metric spaces  the various definitions of Sobolev spaces are equivalent, but also that various definitions of $|Du|$ are equivalent, including the Cheeger gradient or the minimal $p$-relaxed slope of $u$. { This identification holds up to sets of $\nu$-measure zero, since elements of the Newton-Sobolev space are defined everywhere and in the other definitions the Sobolev functions are defined $\nu$-a.e.}

Recall that for a function $u : \mathbb{X} \rightarrow \R$, its {\it slope} (also called local Lipschitz constant) is defined by
$$\vert \nabla u \vert(x) := \limsup_{y \to x} \frac{\vert u(y) - u(x)\vert}{d(x,y)},$$
with the convention that  $\vert \nabla u \vert(x) = 0$ if $x$ is an isolated point.
	
\begin{remark}\label{curvature} {\rm  Obviously, for locally Lipschitz functions, $\vert Du \vert \leq \vert \nabla u \vert$. In general the equality is not true, but there are two important cases in which we have $\vert Du \vert = \vert \nabla u \vert$ $\nu$-a.e. These are:
			
(1) When $\nu$ is  doubling and $(\mathbb{X},d, \nu)$ supports a weak $(1,p)$-Poincar\'e inequality for some $p > 1$ (see \cite{Chegeer});
			
(2) When $(\mathbb{X},d, \nu)$ is a metric measure spaces with Riemannian Ricci curvature bounded from below { (see \cite{AGS2,AGS3})}. $\blacksquare$}
\end{remark}

{ It is possible that the $p$-weak upper gradient depends on $p$, even if the function is Lipschitz. See for instance \cite{dMS}, where the authors show that the Poincar\'e inequality is satisfied only for some range of $p$, so the $p$-weak upper gradient agrees with the slope for sufficiently large $p$ and is identically zero for small $p$.}

As in the case of Sobolev functions, in the literature there are several different ways to characterise the total variation in metric measure spaces. However, on complete and separable metric spaces equipped with a doubling measure (or even under a bit weaker assumptions), these notions turn out to be equivalent, see \cite{ADiM} and \cite{DiMarinoTh}. In this paper, we will employ the definition of total variation introduced by Miranda in \cite{Miranda1}.  For $u \in L^1(\mathbb{X},\nu)$, we define the total variation of $u$ on an open set $\Omega \subset \mathbb{X}$ by the formula
\begin{equation}\label{dfn:totalvariationonmetricspaces}
\vert D u \vert_{\nu}(\Omega):= \inf \left\{ \liminf_{n \to \infty} \int_\Omega g_{u_n} \, d\nu \ : \ u_n \in Lip_{loc}( \Omega), \ u_n \to u \ \hbox{in} \ L^1(\Omega, \nu) \right\},
\end{equation}
where $g_{u_n}$ is a $1$-weak upper gradient of $u$ (we may take $g_{u_n} = |\nabla u_n|$, see \cite{ADiM}). Under the assumptions that $\nu$ is doubling on $\Omega$ and $(\Omega,d,\nu)$ satisfies a weak $(1,1)$-Poincar\'e inequality, since by \cite[Theorem 5.1]{BjBj} Lipschitz functions are dense in $W^{1,1}(\Omega,d,\nu)$, in the definition above we may require that $u_n$ are Lipschitz functions instead of locally Lipschitz functions. Moreover, the total variation $|Du|_\nu(\mathbb{X})$ defined by formula \eqref{dfn:totalvariationonmetricspaces} is lower semicontinuous with respect to convergence in $L^1(\mathbb{X},\nu)$.

The space of functions of bounded variation $BV(\mathbb{X},d,\nu)$ consists of all functions $u \in L^1(\mathbb{X},\nu)$ such that $|Du|_\nu(\mathbb{X}) < \infty$. It is a Banach space with respect to the norm
$$\Vert u \Vert_{BV(\mathbb{X},d,\nu)}:= \Vert u \Vert_{L^1(\mathbb{X},\nu)} + \vert D u \vert_{\nu}(\mathbb{X}).$$
Convergence in norm is often too much to ask when we deal with $BV$ functions, therefore we will employ the notion of strict convergence. We say that a sequence $\{ u_i \} \subset BV(\mathbb{X},d,\nu)$ {\it strictly converges} to $u \in BV(\mathbb{X},d,\nu)$, if $u_i \to  u$ in $L^1(\mathbb{X}, \nu)$ and $| Du_i |_\nu(\mathbb{X}) \to | Du |_\nu(\mathbb{X})$.

\subsection{The differential structure}\label{subsec:giglistructure}

The main tool we will employ in order to provide a characterisation of the subdifferential of the Cheeger energy $\mathsf{Ch}_p$ is the first-order differential structure on metric measure spaces introduced by Gigli. We follow the presentation made by its author in \cite{Gig} (for $p = 2$) and by Buffa-Comi-Miranda in \cite{BCM} (for arbitrary $p \in [1,\infty]$).  From now on, we assume that $\mathbb{X}$ is a complete and separable metric space and $\nu$ is a nonnegative Radon measure.

\begin{definition}{\rm
We define the cotangent module to $\mathbb{X}$ as
$$ \mbox{PCM}_p = \left\{ \{(f_i, A_i)\}_{i \in \N} \ : \ (A_i)_{i \in\N} \subset \mathcal{B}(\mathbb{X}), \  f_i \in D^{1,p}(A_i), \ \ \sum_{i \in \N} \int_{A_i} |Df_i|^p \, d\nu < \infty  \right\},$$
where $A_i$ is a partition of $\mathbb{X}$. We define the equivalence relation $\sim$ as
$$\{(A_i, f_i)\}_{i \in \N} \sim \{(B_j,g_j)\}_{j \in \N} \quad \mbox{if} \quad |D(f_i - g_j)| = 0 \ \ \nu-\hbox{a.e. on} \ A_i \cap B_j.$$
Consider the map $ \vert \cdot \vert_*: \mbox{PCM}_p/\sim \rightarrow L^p(\mathbb{X}, \nu)$ given by
$$\vert \{(f_i, A_i)\}_{i \in \N} \vert_* := \vert D f_i \vert$$
$\nu$-everywhere on $A_i$ for all $i \in \N$, namely the {\it pointwise norm} on $ \mbox{PCM}_p/\sim$.

In $ \mbox{PCM}_p/\sim$ we define the norm $\| \cdot \|$ as
$$ \|  \{(f_i, A_i)\}_{i \in \N} \|^p = \sum_{i\in \N} \int_{A_i}|Df_i|^p $$
and set $L^p(T^* \mathbb{X})$ to be the closure of $ \mbox{PCM}_p/\sim$ with respect to this norm, i.e. we identify functions which differ by a constant and we identify possible rearranging of the sets $A_i$. $L^p(T^* \mathbb{X})$ is called the {\it cotangent module} and its elements will be called {\it $p$-cotangent vector field}.

$L^p(T^* \mathbb{X})$ is a $L^p(\nu)$-normed module, we denote by  $L^q(T\mathbb{X})$ the dual module of $L^p(T^* \mathbb{X})$, namely $L^q(T\mathbb{X}):= \hbox{HOM}(L^p(T^* \mathbb{X}), L^1(\mathbb{X}, \nu))$, which is a $L^q(\nu)$-normed module. The elements of $L^q(T\mathbb{X})$ will be called  {\it $q$-vector fields} on $\mathbb{X}$. The duality between $\omega \in L^p(T^* \mathbb{X})$ and $L \in  L^q(T\mathbb{X})$ will be denoted by $\omega(X) \in L^1(\mathbb{X}, \nu)$. Since the module $L^p(T^* \mathbb{X})$ is reflexive we can identify
$$ L^q(T\mathbb{X})^* = L^p(T^*\mathbb{X}),$$
where $\frac{1}{p} + \frac{1}{q} = 1$.}

\end{definition}
	
\begin{definition}\label{dfn:differential}
{\rm
Given $f \in D^{1,p}(\mathbb{X})$ we can define its {\it differential} $df$ as an element of $L^p(T^* \mathbb{X})$ given by the formula $df = (f, \mathbb{X})$.}
\end{definition}

Clearly, the operation of taking the differential is linear as an operator from $D^{1,p}(\mathbb{X})$ to $L^p(T^{*} \mathbb{X})$; moreover, from the definition of the norm in $L^p(T^{*} \mathbb{X})$ it is clear that this operator is bounded with norm equal to one. Moreover, again from the definition of the pointwise norm, it is clear that
\begin{equation*}
|df|_* = |Df| \qquad \nu\mbox{-a.e. on } \mathbb{X} \mbox{ for all } f \in W^{1,p}(\mathbb{X},d,\nu).
\end{equation*}

\begin{definition}\label{dfn:gradient}
{\rm
Given $f \in W^{1,p}(\mathbb{X},d,\nu)$, we say that $X \in L^q(T\mathbb{X})$ is its {\it $p$-gradient} provided that
\begin{equation*}
df(X) = |X|^q = |df|_*^p \qquad \nu-\mbox{a.e. on } \mathbb{X}.
\end{equation*}
The set of all $p$-gradients of $f$ will be denoted by $\mbox{Grad}_p(f)$.}
\end{definition}

Notice that for any $X \in L^q(T\mathbb{X})$ and $f \in W^{1,p}(\mathbb{X},d,\nu)$ we have
\begin{equation}
df(X) \leq |df|_* |X| \leq \frac1p |df|_*^p + \frac1q |X|^q.
\end{equation}
 Hence,
\begin{equation*}
X \in \mbox{Grad}_p(f) \Leftrightarrow \int_{\mathbb{X}} df(X) \, d\nu \geq \frac1p \int_{\mathbb{X}} |df|_*^p \, d\nu + \frac1q \int_{\mathbb{X}} |X|^q \, d\nu.
\end{equation*}
The name $p$-gradient is a bit misleading: in Euclidean spaces, a $p$-gradient of $u \in W^{1,p}(\mathbb{X},d,\nu)$ in coordinates is $X = |\nabla u|^{p-2} \nabla u$. Nonetheless, we will use this name in order to keep the notation consistent with \cite{Gig}, where it was introduced for $p = 2$.

Now, we define the divergence of a vector field, in the case when it can be represented by an $L^q$ function.  Following \cite{BCM,DiMarinoTh}, we set
$$ \mathcal{D}^q(\mathbb{X}) = \left\{ X \in L^q(T\mathbb{X}): \,\, \exists f \in L^q(\mathbb{X},\nu) \,  \,\, \int_\mathbb{X} fg d\nu = - \int_{\mathbb{X}} dg(X) d\nu \ \ \forall g \in  W^{1,p}(\mathbb{X},d,\nu) \right\}. $$
Here, the right hand side makes sense as an action of an element of $L^p(T^* \mathbb{X})$ on an element of $L^{q}(T\mathbb{X})$; the resulting function is an element of $L^1(\mathbb{X},\nu)$. The function $f$, which is unique by the density of  $W^{1,p}(\mathbb{X},d,\nu)$ in $ L^{p}(\mathbb{X},\nu)$, will be called the {\it $q$-divergence} of the vector field $X$, and we shall write $\mbox{div}(X)= f$.  An exhaustive discussion on the uniqueness of the divergence and its dependence on $q$ can be found in \cite{BCM}.

In the course of the paper, in order to study the subdifferential of the Cheeger energy and the total variation, we will need to consider the case when $X \in L^q(T\mathbb{X})$ for $q \in (1,\infty]$, but its divergence is in $L^2(\mathbb{X},\nu)$. To this end, we introduce the following definition (compare to the definition of derivations with integrable divergence in \cite{BCM,DiMarinoTh}). For $\frac1r + \frac1s = 1$, we set
$$ \mathcal{D}^{q,r}(\mathbb{X}) = \bigg\{ X \in L^q(T\mathbb{X}): \,\, \exists f \in L^r(\mathbb{X},\nu) \quad \forall g \in  W^{1,p}(\mathbb{X},d,\nu) \cap L^s(\mathbb{X},\nu) \qquad\qquad\qquad $$
$$ \qquad\qquad\qquad\qquad\qquad\qquad\qquad\qquad\qquad\qquad \int_\mathbb{X} fg \, d\nu = - \int_{\mathbb{X}} dg(X) \, d\nu \bigg\}. $$
The function $f$, which is unique by the density of $W^{1,p}(\mathbb{X},d,\nu)$ in $ L^{p}(\mathbb{X},\nu)$, will be called the $(q,r)$-divergence of $X$. We will still write $\mbox{div}(X) = f$ when it is clear from the context. We will be most interested in the case $r = s = 2$. Also, note that $\mathcal{D}^{q,q}(\mathbb{X}) = \mathcal{D}^{q}(\mathbb{X})$.  Furthermore, whenever Lipschitz functions are dense in $W^{1,p}(\mathbb{X},d,\nu)$, see for instance \cite[Theorem 5.1]{BB}, then the divergence does not depend on $r$ in the following sense: if $f$ is the $(q,r)$-divergence of $X$ and $f \in L^{r'}(\mathbb{X},\nu)$, then it is also the $(q,r')$-divergence of $X$.

It is easy to see that given $X \in \mathcal{D}^q(\mathbb{X})$ and $f \in L^\infty(\mathbb{X},\nu) \cap D^{1,p}(\mathbb{X})$ with $\vert D f \vert \in L^\infty(\mathbb{X},\nu)$, we have
$$fX \in \mathcal{D}^q(\mathbb{X}) \quad \hbox{and} \quad \mbox{div}(fX) = df(X) + f \mbox{div}(X).$$
 Furthermore, whenever $X \in \mathcal{D}^{q,r}(\mathbb{X})$ and $f \in \mbox{Lip}(\mathbb{X})$ has compact support, we have
$$fX \in \mathcal{D}^{q,r}(\mathbb{X}) \quad \hbox{and} \quad \mbox{div}(fX) = df(X) + f \mbox{div}(X).$$

Finally, let us note that when the metric measure space is Euclidean equipped with the Lebesgue measure, the vector fields and differentials arising from this construction coincide with their standard counterparts defined in coordinates, see \cite[Remark 2.2.4]{Gig}. Therefore, in this case all the definitions and results obtained in this paper can be read in coordinates and are compatible with their Euclidean counterparts.

\section{The $p$-Laplacian evolution equation}\label{sec:plaplace}

In this Section, we apply the framework presented in the previous Section to study the $p$-Laplacian evolution equation, i.e. the gradient flow of the Cheeger energy $\mathsf{Ch}_p$ with $1 < p < \infty$. Because we need to work in a Hilbert space, we will study the problem in the space $W^{1,p}(\mathbb{X},d,\nu) \cap L^2(\mathbb{X},\nu)$.  We again assume that $(X,d)$ is complete and separable and that $\nu$ is a nonnegative measure which is finite on bounded sets.

The Cheeger energy (restricted to $L^2(\mathbb{X},\nu)$) $\mathsf{Ch}_p: L^2(\mathbb{X},\nu) \rightarrow [0,+\infty]$ is defined by the formula
\begin{equation}
\mathsf{Ch}_p(u) = \twopartdef{\displaystyle\frac{1}{p} \displaystyle\int_{\mathbb{X}} |Du|^p \, d\nu}{u \in W^{1,p}(\mathbb{X},d,\nu) \cap L^2(\mathbb{X},\nu) }{+ \infty}{u \in L^2(\mathbb{X},\nu) \setminus W^{1,p}(\mathbb{X},d,\nu).}
\end{equation}
As in the previous Section, we will provide a notion of solutions to the gradient flow associated to $\mathsf{Ch}_p$ and prove existence and uniqueness of solutions. In other words, we study the abstract Cauchy problem
\begin{equation}\label{eq:cauchylp}
\left\{ \begin{array}{ll} u'(t) + \partial \mathsf{Ch}_p(u(t)) \ni 0, \quad t \in [0,T] \\[5pt] u(0) = u_0. \end{array}\right.
\end{equation}
Again, our approach is to express the subdifferential of $\mathsf{Ch}_p$ in terms of the differential structure by Gigli. We define the following operator.

\begin{definition}{\rm $(u,v) \in \mathcal{A}_p$ if and only if $u, v \in L^2(\mathbb{X}, \nu)$, $u \in W^{1,p}(\mathbb{X}, d, \nu)$ and there exists a vector field $X \in \mathcal{D}^{q,2}(\mathbb{X})$ with  $| X |^q \leq | du |_*^p$ $\nu$-a.e. such that the following conditions hold:
\begin{equation}
-\mbox{div}(X) = v \quad \hbox{in } \mathbb{X};
\end{equation}
\begin{equation}
du(X) = |du|_*^p  \quad \nu\hbox{-a.e. in } \mathbb{X}.
\end{equation}
}
\end{definition}

\begin{theorem}\label{thm:plaplaceflow}
$\partial \mathsf{Ch}_p = \mathcal{A}_p$. Furthermore, the operator $\mathcal{A}_p$ is completely accretive and the domain of $\mathcal{A}_p$ is dense in $L^2(\mathbb{X}, \nu)$.
\end{theorem}

\begin{proof}
First, let us see that $\mathcal{A}_p \subset \partial \mathsf{Ch}_p$. Let $(u,v) \in \mathcal{A}_p$. Then, given $w \in W^{1,p}(\mathbb{X}, d, \nu)$, we have
$$ \int_{\mathbb{X}} v(w-u) \, d \nu = - \int_{\mathbb{X}}  \mbox{div}(X)(w -u) \, d\nu = \int_{\mathbb{X}} d(w-u)(X) \, d\nu $$
$$ = \int_{\mathbb{X}} dw(X) \, d\nu - \int_{\mathbb{X}} du(X) \, d\nu \leq \frac{1}{p} \int_{\mathbb{X}} |dw|_*^p \, d\nu + \frac{1}{q} \int_{\mathbb{X}} |X|^q \, d\nu - \int_{\mathbb{X}} |du|_*^p \, d\nu $$
$$ \leq \frac{1}{p} \int_{\mathbb{X}} |dw|_*^p  \, d\nu - \frac{1}{p} \int_{\mathbb{X}} |du|_*^p \, d\nu = \mathsf{Ch}_p(w) - \mathsf{Ch}_p(u).$$
If $w \notin W^{1,p}(\mathbb{X},d,\nu)$, then $\mathsf{Ch}_p(w) = +\infty$ and this inequality also holds. Consequently, $(u,v) \in \partial \mathsf{Ch}_p$. { Notice that since $\mathcal{A}_p \subset \partial \mathsf{Ch}_p$, the operator $\mathcal{A}_p$ is monotone.}

Since $\mathsf{Ch}_p$ is convex and lower semicontinuous (see \cite{AGS1}), the operator $\partial \mathsf{Ch}_p$ is maximal monotone. So, if we show that  $\mathcal{A}_p$ satisfies the range condition, by Minty Theorem we would also have that the operator $\mathcal{A}_p$ is maximal monotone, and consequently $\partial \mathsf{Ch}_p = \mathcal{A}_p$. In order to finish the proof, let us see that $\mathcal{A}_p$ satisfies the range condition, i.e.
\begin{equation}
\hbox{Given} \ g \in L^2(\mathbb{X}, \nu), \ \exists \, u \in D(\mathcal{A}_p) \ s.t. \ \  g \in u + \mathcal{A}_p(u).
\end{equation}
We rewrite it as
$$ g \in u + \mathcal{A}_p(u) \iff (u, g-u) \in \mathcal{A}_p,$$
so we need to show that there exists a vector field  $X \in \mathcal{D}^{q,2}(\mathbb{X})$ with  $| X |^q \leq | du |_*^p$ $\nu$-a.e. such that the following conditions hold:
\begin{equation}\label{eq:rangecondition1p}
-\mbox{div}(X) = g-u \quad \hbox{in } \mathbb{X};
\end{equation}
\begin{equation}\label{eq:rangecondition2p}
du(X) = |du|_*^p \quad \nu \hbox{-a.e. in } \mathbb{X}.
\end{equation}
Again, we are going to prove this by means of the Fenchel-Rockafellar duality theorem. We set
$$U = W^{1,p}(\mathbb{X},d,\nu) \cap L^2(\mathbb{X},\nu), \qquad V = L^p(T^{*} \mathbb{X}),$$ and the operator $A: U \rightarrow V$ is defined by the formula
$$A(u) = du,$$
where $du$ is the differential of $u$ in the sense of Definition \ref{dfn:differential}. Hence, $A$ is a linear and continuous operator. Moreover, the dual spaces to $U$ and $V$ are
$$ U^* = (W^{1,p}(\mathbb{X},d,\nu) \cap L^2(\mathbb{X},\nu))^*, \qquad V^* =  L^q(T\mathbb{X}).$$
We set $E: L^p(T^{*} \mathbb{X}) \rightarrow \mathbb{R}$ by the formula
\begin{equation}
E(v) = \frac1p \int_{\mathbb{X}} |v|_*^p \, d\nu.
\end{equation}
It is clear that the functional $E^*:  L^q(T \mathbb{X}) \rightarrow [0,\infty]$ is given by the formula
$$E^*(v^*) = \frac1q \int_{\mathbb{X}} |v^*|^q \, d\nu.$$
We also set $G: W^{1,p}(\mathbb{X},d,\nu) \cap L^2(\mathbb{X},\nu) \rightarrow \mathbb{R}$ by
$$G(u):= \frac12 \int_{\mathbb{X}} u^2 \, d\nu - \int_{\mathbb{X}} ug \, d\nu.$$
The functional $G^* : (W^{1,p}(\mathbb{X},d,\nu) \cap L^2(\mathbb{X},\nu))^* \rightarrow [0,+\infty ]$ is given by
$$G^*(u^*) = \displaystyle\frac12 \int_{\mathbb{X}} (u^*  + g)^2 \, d\nu.$$
Now,  for fixed $v^* \in L^q(T\mathbb{X})$ in the domain of $A^*$, and any $u \in W^{1,p}(\mathbb{X},d,\nu) \cap L^2(\mathbb{X},\nu)$, we have
$$\int_{\mathbb{X}} u \, (A^* v^*) \, d\nu = \langle u, A^* v^* \rangle  = \langle v^*, Au \rangle = \int_{\mathbb{X}} du(v^*) \, d\nu,$$
so the definition of the divergence of $v^*$ is satisfied with
\begin{equation}
\mbox{div}(v^*) = - A^* v^*.
\end{equation}
In particular, $\mbox{div}(v^*) \in L^2(\mathbb{X},\nu)$. In other words, the domain of $A^*$ is $\mathcal{D}^{q,2}(\mathbb{X}).$

Consider the energy functional $\mathcal{G}_p: L^2(\mathbb{X}, \nu) \rightarrow (-\infty, + \infty]$ defined by
\begin{equation}
\mathcal{G}_p(u) := \mathsf{Ch}_p(u) + G(u).
\end{equation}
Since $\mathcal{G}_p$ is coercive, convex and lower semi-continuous, the primal minimization problem
$$\min_{u \in L^2(\mathbb{X}, \nu)} \mathcal{G}_p(u)$$
admits an optimal solution $\overline{u} \in W^{1,p}(\mathbb{X},d,\nu)$. Notice that we may write it a
\begin{equation}\label{eq:pnotequal2primal}
\min_{u \in L^2(\mathbb{X}, \nu)} \mathcal{G}_p(u) = \inf_{u \in W^{1,p}(\mathbb{X},d,\nu) \cap L^2(\mathbb{X},\nu)} \bigg\{ E(Au) + G(u) \bigg\}.
\end{equation}
Hence, its dual problem is
\begin{equation}\label{eq:pnotequal2dual}
\sup_{v^* \in L^q(T\mathbb{X})} \bigg\{  - E^*(-v^*) - G^*(A^* v^*) \bigg\}.
\end{equation}
For $u_0 \equiv 0$ we have $E(Au_0) = 0 < \infty$, $G(u_0) = 0 < \infty$ and $E$ is continuous at $0$. By the Fenchel-Rockafellar duality theorem, we have
\begin{equation}
\inf \eqref{eq:pnotequal2primal} = \sup \eqref{eq:pnotequal2dual}
\end{equation}
and
\begin{equation}
\hbox{the dual problem \eqref{eq:pnotequal2dual} admits at least one solution $\overline{v}$.}
\end{equation}
Now, we use the extremality conditions between the solutions $\overline{u}$ of the primal problem and $\overline{v}$ of the dual problem, see \cite[Chapter III]{EkelandTemam}. These are
\begin{equation}
E(A\overline{u}) + E^*(-\overline{v}^*) = \langle -\overline{v}^*, A\overline{u} \rangle
\end{equation}
\begin{equation}\label{eq:extremalitycondition2}
G(\overline{u}) + G^*(A^* \overline{v}^*) = \langle \overline{u}, A^* \overline{v}^* \rangle.
\end{equation}
The first condition immediately gives us
\begin{equation*}
\frac1p \int_\mathbb{X} |d\overline{u}|_*^p \, d\nu + \frac1q \int_\mathbb{X} |-\overline{v}^*|^q \, d\nu = \int_{\mathbb{X}} d\overline{u}(- \overline{v}^*) \, d\nu,
\end{equation*}
so $-\overline{v}^* \in \mathrm{Grad}_p(\overline{u})$.  In particular, $|-\overline{v}^*|^q = |d\overline{u}|_*^p$ $\nu$-a.e.

From the second condition, using the definition of the convex conjugate, we have that for any $w \in L^2(\mathbb{X},\nu)$
\begin{equation*}
G^*(A^* \overline{v}^*) \geq \langle w, A^* \overline{v}^* \rangle - G(w),
\end{equation*}
so by equation \eqref{eq:extremalitycondition2} we have
\begin{equation*}
G(w) - G(\overline{u}) \geq \langle w, A^* \overline{v}^* \rangle - G^*(A^* \overline{v}^*) - G(\overline{u}) \geq \langle w - \overline{u}, A^* \overline{v}^* \rangle,
\end{equation*}
so $A^* \overline{v}^* \in \partial G(\overline{u})$. On the other hand, $\partial G(\overline{u}) = \{ \overline{u} - g \}$, so
\begin{equation}
-\mathrm{div}(\overline{v}^*) = \overline{u} - g.
\end{equation}
Hence, the pair $(\overline{u}, -\overline{v}^*)$ satisfies the desired conditions \eqref{eq:rangecondition1p}-\eqref{eq:rangecondition2p}. Hence, the operator $\mathcal{A}_p$ satisfies the range condition, so it is maximal monotone; hence, $\mathcal{A}_p = \partial \mathsf{Ch}_p$.

Let $P_{0}$ denote the set of all functions $T\in C^{\infty}(\R)$ satisfying $0\le T'\le 1$ such that $T'$  is compactly supported, and $x=0$ is not contained in the support $\textrm{supp}(T)$ of $T$. To prove that $\mathcal{A}_p$ is  a completely accretive operator we must show that  (see \cite{ACMBook,BCr2})
$$ \int_{\mathbb{X}}T(u_1-u_2)(v_1-v_2)\, d\nu \geq 0 $$
for every $T\in P_{0}$ and every $(u_i,v_i) \in \mathcal{A}_p$, $i=1,2$. In fact, given $(u_i, v_i) \in \mathcal{A}_p$, $i =1,2$, we have that  there exist  vector fields  $X_i \in \mathcal{D}^{q,2}(\mathbb{X})$  with  $| X_i |^q \leq | du_i |_*^p$ $\nu$-a.e. such that the following conditions hold:
\begin{equation}\label{e1caplaplace}
-\mbox{div}(X_i) = v_i \quad \hbox{in } \mathbb{X};
\end{equation}
\begin{equation}\label{e2caplaplace}
du_i(X) = |du_i|_*^p  \quad \nu\hbox{-a.e. in } \mathbb{X}.
\end{equation}
Then, since $T(u_1 - u_2) \in W^{1,p}(\mathbb{X}, d, \nu)$, using \eqref{e1caplaplace} and \eqref{e2caplaplace} and applying the chain rule \cite[Corollary 2.2.8]{Gig} (given there for $p = 2$, but the proof works for all $1 < p < \infty$) we get
$$ \int_{\mathbb{X}}T(u_1-u_2)(v_1-v_2)\, d\nu = -  \int_{\mathbb{X}} T(u_1-u_2) (\mbox{div}(X_1) - \mbox{div}(X_2)) \, d\nu $$
$$= \int_{\mathbb{X}} dT(u_1-u_2)(X_1 - X_2) \, d \nu  =  \int_{\mathbb{X}} T'(u_1-u_2) d(u_1-u_2)(X_1 - X_2) \, d \nu  $$
$$= \int_{\mathbb{X}} T'(u_1-u_2) du_1(X_1) \, d\nu -  \int_{\mathbb{X}} T'(u_1-u_2) du_1(X_2) \, d\nu $$
$$- \int_{\mathbb{X}} T'(u_1-u_2) du_2(X_1) \, d\nu +\int_{\mathbb{X}} T'(u_1-u_2)du_2(X_2) \, d\nu $$
$$= \int_{\mathbb{X}} T'(u_1-u_2)  |du_1|_*^p \, d\nu -  \int_{\mathbb{X}} T'(u_1-u_2) du_1(X_2) \, d\nu $$
$$- \int_{\mathbb{X}} T'(u_1-u_2) du_2(X_1) \, d\nu + \int_{\mathbb{X}} T'(u_1-u_2) |du_2|_*^p \, d\nu $$
$$\geq \int_{\mathbb{X}} T'(u_1-u_2)  |du_1|_*^p \, d \nu - \frac1p \int_{\mathbb{X}} T'(u_1-u_2)  |du_1|_*^p \, d\nu  - \frac1q \int_{\mathbb{X}} T'(u_1-u_2) \vert X_2 \vert^q \, d \nu $$
$$- \frac1p \int_{\mathbb{X}} T'(u_1-u_2)  |du_2|_*^p \, d\nu  - \frac1q \int_{\mathbb{X}} T'(u_1-u_2) \vert X_1\vert^q \, d \nu + \int_{\mathbb{X}} T'(u_1-u_2) |du_2|_*^p \, d \nu  $$
$$\geq \frac1q \int_{\mathbb{X}} T'(u_1-u_2)  |du_1|_*^p \, d \nu - \frac1q \int_{\mathbb{X}} T'(u_1-u_2) \vert X_1\vert^q \, d \nu$$
$$+\frac1q \int_{\mathbb{X}} T'(u_1-u_2)  |du_2|_*^p \, d \nu - \frac1q \int_{\mathbb{X}} T'(u_1-u_2) \vert X_2\vert^q \, d \nu \geq 0, $$
since $T' \geq 0$ and $| X_i |^q \leq | du_i |_*^p$ $\nu$-a.e. for $i=1,2$. Hence, $\mathcal{A}_p$ is completely accretive.

Finally, by \cite[Proposition 2.11]{Brezis}, we have
$$ D(\partial \mathsf{Ch}_p) \subset  D(\mathsf{Ch}_p) =  W^{1,p}(\mathbb{X},d,\nu) \cap L^2(\mathbb{X},\nu) \subset \overline{D(\mathsf{Ch}_p)}^{L^2(\mathbb{X}, \nu)} \subset \overline{D(\partial \mathsf{Ch}_p)}^{L^2(\mathbb{X}, \nu)},$$
from which follows the density of the domain.
\end{proof}

{
\begin{remark}\label{CA}{\rm Let us point out that for $p=2$ and an additional mild regularity assumption on the space, the characterisation given in Theorem \ref{thm:plaplaceflow} was obtained by Gigli in \cite[Proposition 2.3.14]{Gig}. Also, complete accretivity of the operator $\partial \mathsf{Ch}_p$ was proved in \cite[Proposition 4.15]{AGS2} for $p = 2$ and in \cite[Proposition 3.2]{Kell} for $1 < p < \infty$; we give the proof for the operator $\mathcal{A}_p$ for the sake of completeness.} $\blacksquare$
\end{remark}
}

Furthermore, we may give a more detailed characterisation of solutions in terms of variational inequalities. We present the equivalent characterisations in the following Corollary.

\begin{corollary}\label{cor:subdifferentialplaplace}
The following conditions are equivalent: \\
$(a)$ $(u,v) \in \partial\mathsf{Ch}_p$; \\
$(b)$ $(u,v) \in \mathcal{A}_p$, i.e. $u, v \in L^2(\mathbb{X}, \nu)$, $u \in W^{1,p}(\mathbb{X}, d, \nu)$ and there exists a vector field  $X \in \mathcal{D}^{q,2}(\mathbb{X})$ with  $| X |^q \leq | du |_*^p$ $\nu$-a.e. such that $-\mbox{div}(X) = v$ in $\mathbb{X}$ and
\begin{equation}
du(X) = |du|_*^p  \quad \nu\hbox{-a.e. in } \mathbb{X};
\end{equation}
$(c)$ $u, v \in L^2(\mathbb{X}, \nu)$, $u \in W^{1,p}(\mathbb{X}, d, \nu)$ and there exists a vector field  $X \in  \mathcal{D}^{q,2}(\mathbb{X})$ with  $| X |^q \leq | du |_*^p$ $\nu$-a.e. such that $-\mbox{div}(X) = v$ in $\mathbb{X}$ and for every $w \in L^2(\mathbb{X},\nu) \cap W^{1,p}(\mathbb{X},d,\nu)$
\begin{equation}\label{eq:variationalinequality}
\int_{\mathbb{X}} v(w-u) \, d\nu \leq \int_{\mathbb{X}} dw(X) \, d\nu - \int_{\mathbb{X}} |du|_*^p \, d\nu;
\end{equation}
$(d)$ $u, v \in L^2(\mathbb{X}, \nu)$, $u \in W^{1,p}(\mathbb{X}, d, \nu)$ and there exists a vector field  $X \in \mathcal{D}^{q,2}(\mathbb{X})$ with  $| X |^q \leq | du |_*^p$ $\nu$-a.e. such that $-\mbox{div}(X) = v$ in $\mathbb{X}$ and for every $w \in L^2(\mathbb{X},\nu) \cap W^{1,p}(\mathbb{X},d,\nu)$
\begin{equation}
\int_{\mathbb{X}} v(w-u) \, d\nu = \int_{\mathbb{X}} dw(X) \, d\nu - \int_{\mathbb{X}} |du|_*^p \, d\nu.
\end{equation}
\end{corollary}

\begin{proof}
The equivalence of $(a)$ and $(b)$ is exactly the content of Theorem \ref{thm:plaplaceflow}. To see that $(b)$ implies $(d)$, multiply the equation $v = -\mbox{div}(X)$ by $w - u$ and integrate over $\mathbb{X}$ with respect to $\nu$. Using the definition of the divergence, we get that
\begin{equation*}
\int_{\mathbb{X}} v(w-u) \, d\nu = - \int_{\mathbb{X}} (w - u) \mbox{div}(X) \, d\nu = \int_{\mathbb{X}} d(w-u)(X) \, d\nu = \int_{\mathbb{X}} dw(X) \, d\nu - \int_{\mathbb{X}} |du|_*^p \, d\nu.
\end{equation*}
It is clear that $(d)$ implies $(c)$. To finish the proof, let us see that $(c)$ implies $(b)$. If we take $w = u$ in \eqref{eq:variationalinequality}, we get
\begin{equation*}
\int_{\mathbb{X}} |du|_*^p \, d\nu \leq \int_{\mathbb{X}} du(X) \, d\nu \leq \int_{\mathbb{X}} |du|_* |X| \, d\nu \leq \int_{\mathbb{X}} \bigg( \frac{1}{p} |du|_*^p + \frac{1}{q} |X|^q \bigg) \, d\nu \leq \int_{\mathbb{X}} |du|_*^p \, d\nu.
\end{equation*}
Hence, $du(X) = |du|_*^p$ $\nu$-a.e. in $\mathbb{X}$.
\end{proof}

In light of the above results, it is natural to introduce the following concept of solution to the gradient flow given by the Cheeger energy $\mathsf{Ch}_p$:

\begin{definition}
We define in $L^2(\mathbb{X},\nu)$ the multivalued operator $\Delta_{p,\nu}$ by
\begin{center}$(u, v) \in \Delta_{p,\nu}$ \ if and only if \ $-v \in \partial\mathsf{Ch}_p(u)$.
\end{center}
\end{definition}

{ We will use the notation $\Delta_{\nu} =\Delta_{2,\nu}$ for the Laplacian.}

{
\begin{remark} {\rm  It should be observed that, in general, the Laplacian $\Delta_\nu$ is not a linear operator: the potential lack of linearity is strictly related to the fact that the space $W^{1,2}(\mathbb{X}, d, \nu)$ needs not be Hilbert. This is the case, for example, on the metric measure space $(\R^N, \Vert \cdot\Vert, \mathcal{L}^N)$
where $\Vert \cdot\Vert$ is any norm not coming from an inner product. The metric measure spaces $(\mathbb{X}, d, \nu)$ for which $W^{1,2}(\mathbb{X}, d, \nu)$ is a Hilbert space (also called infinitesimally Hilbertian spaces) were studied by Gigli in \cite{Gig2}, where are proved several characterizations of these spaces, among others that the Laplacian $\Delta_\nu$ is a linear operator.} $\blacksquare$
\end{remark}

\begin{remark}\label{Trivial}
{\rm Although the content of this section makes sense in a general
metric measure space, as was point out by Ambrosio et al. \cite[Remark 4.12]{AGS2}, if no additional assumption is made it may happen that the constructions presented here are trivial. For instance, choose any sequence $\{a_n \}$ of positive real number such that $\sum_{n=0}^\infty a_n <\infty$. Let $\{q_n \ : \ n \in \N\}$ be an enumeration of the rational numbers. Consider the Borel measure $\nu$ in $\R$ defined by
$$\nu:= \sum_{n=1}^\infty a_n \delta_{q_n},\quad \hbox{where $\delta_{q_n}$ is the Dirac measure at $q_n$},$$
Therefore $W^{1,2}(\R, d_{Eucl}, \nu)= L^2(\R, \nu)$ and all its elements have null minimal weak upper gradient (\cite[Remark 4.12]{AGS2}). Then
$$\mathsf{Ch}_2(u) =0, \qquad \forall \, u \in W^{1,2}(\R, d_{Eucl}, \nu)$$
 and the corresponding gradient flow is trivial.} $\blacksquare$
\end{remark}
}

We have that the abstract Cauchy problem \eqref{eq:cauchylp} corresponds to the Cauchy problem for the $p$-Laplacian, i.e.,
\begin{equation}\label{eq:p-Lapcauchy}
\left\{ \begin{array}{ll} \partial_t u(t) \in  \Delta_{p,\nu}(u(t)), \quad t \in [0,T] \\[5pt] u(0) = u_0. \end{array}\right.
\end{equation}

\begin{definition}\label{dfn:plaplaceflow}
{\rm  Given $u_0 \in L^2(\mathbb{X},\nu)$, we say that $u$ is a {\it weak solution} of the Cauchy problem \eqref{eq:p-Lapcauchy} in $[0,T]$, if { $u \in C([0,T];L^2(\mathbb{X},\nu)) \cap W_{loc}^{1,2}(0, T; L^2(\mathbb{X},\nu))$}, $u(0, \cdot) = u_0$, and for almost all $t \in (0,T)$
\begin{equation}
u_t(t, \cdot) \in \Delta_{p,\nu} u(t, \cdot).
\end{equation}
In other words, if $u(t) \in W^{1,p}(\mathbb{X}, d, \nu)$ and there exist vector fields  $X(t) \in  \mathcal{D}^{q,2}(\mathbb{X})$ with  $| X(t) |^q \leq | du(t) |_*^p$ $\nu$-a.e. such that for almost all $t \in [0,T]$ the following conditions hold:
$$ \mbox{div}(X(t)) = u_t(t, \cdot) \quad \hbox{in } \mathbb{X}; $$
$$ du(t)(X(t)) = |du(t)|_*^p \quad \nu\hbox{-a.e. in } \mathbb{X}.$$
}
\end{definition}

Then, { by the Brezis-Komura Theorem (Theorem \ref{BKTheorem}),} as a consequence of Theorem \ref{thm:plaplaceflow}, we have the following existence and uniqueness theorem.  Here, the comparison principle is consequence of the complete accretivity of the operator $\mathcal{A}_p$.

\begin{theorem}\label{ExisUniqp}
For any $u_0 \in L^2(\mathbb{X}, \nu)$ and all $T > 0$ there exists a unique weak solution $u(t)$ of the Cauchy problem \eqref{eq:p-Lapcauchy} in $[0,T]$, with $u(0) =u_0$.  Moreover, the following comparison principle holds: if $u_1, u_2$ are weak solutions for the initial data $u_{1,0}, u_{2,0} \in  L^2(\mathbb{X}, \nu) \cap  L^r(\mathbb{X}, \nu)$, respectively, then
\begin{equation}\label{CompPrincipleplaplace}
\Vert (u_1(t) - u_2(t))^+ \Vert_r \leq \Vert ( u_{1,0}- u_{2,0})^+ \Vert_r \quad \hbox{for all} \ 1 \leq r \leq \infty.
\end{equation}
\end{theorem}

{ The definition of the p-Laplacian and the gradient flow is consistent with the one given by Ambrosio, Gigli and Savar\'e in \cite{AGS1} in terms of the subdifferential of the Cheeger energy $\mathsf{Ch}_p$; what we did is give a precise characterisation of this subdifferential. The definition in \cite{AGS1} also includes also a choice of the element with minimal norm in the subdifferential, but it is a standard property of gradient flows of convex functionals (see \cite{Brezis}). Because we have existence and uniqueness of solutions for both definitions, the two notions of solutions to corresponding gradient flows coincide. Some properties of the gradient flow are listed in \cite[Proposition 6.6]{AGS1}. Let us note that these may be also proved directly using Definition \ref{dfn:plaplaceflow}: for instance, when $\nu(X) < \infty$ we have
\begin{equation*}
\int_{\mathbb{X}} u_t(t,\cdot) \, d\nu = \int_{\mathbb{X}} 1 \, \mathrm{div}(X(t)) \, d\nu = \int_{\mathbb{X}} d1(X) \, d\nu = 0,
\end{equation*}
so the gradient flow of $\mathsf{Ch}_p$ preserves mass.



}


As a direct consequence of Corollary \ref{cor:subdifferentialplaplace}, we also get the following characterisation of weak solutions in terms of variational inequalities.

\begin{corollary}
The following conditions are equivalent: \\
$(a)$ $u$ is a weak solution of the Cauchy problem \eqref{eq:p-Lapcauchy}; \\
$(b)$  { $u \in  C([0,T];L^2(\mathbb{X},\nu)) \cap W_{loc}^{1,2}(0, T; L^2(\mathbb{X},\nu))$}, $u(0, \cdot) = u_0$, $u(t) \in W^{1,p}(\mathbb{X}, d, \nu)$ and there exist vector fields  $X(t) \in  \mathcal{D}^{q,2}(\mathbb{X})$ with $| X(t) |^q \leq | du |_*^p$ $\nu$-a.e. such that for almost all $t \in [0,T]$ we have $\mbox{div}(X(t)) = u_t(t, \cdot)$ in $\mathbb{X}$ and
\begin{equation*}
\int_{\mathbb{X}} u_t (u(t) - w) \, d\nu \leq \int_{\mathbb{X}} dw(X(t)) \, d\nu - \int_{\mathbb{X}} |du(t)|_*^p \, d\nu,\quad \forall \, w \in L^2(\mathbb{X},\nu) \cap W^{1,p}(\mathbb{X},\nu).
\end{equation*}
$(c)$  { $u \in C([0,T];L^2(\mathbb{X},\nu)) \cap W_{loc}^{1,2}(0, T; L^2(\mathbb{X},\nu))$}, $u(0, \cdot) = u_0$, $u(t) \in W^{1,p}(\mathbb{X}, d, \nu)$ and there exist vector fields $X(t) \in  \mathcal{D}^{q,2}(\mathbb{X})$ with $| X(t) |^q \leq | du |_*^p$ $\nu$-a.e. such that for almost all $t \in [0,T]$ we have $\mbox{div}(X(t)) = u_t(t, \cdot)$ in $\mathbb{X}$ and
\begin{equation*}
\int_{\mathbb{X}} u_t (u(t) - w) \, d\nu = \int_{\mathbb{X}} dw(X(t)) \, d\nu - \int_{\mathbb{X}} |du(t)|_*^p \, d\nu,\quad \forall \, w \in L^2(\mathbb{X},\nu) \cap W^{1,p}(\mathbb{X},\nu).
\end{equation*}

\end{corollary}

In particular, weak solutions satisfy the variational inequality
\begin{equation*}
\int_{\mathbb{X}} u_t (u-w) \, d\nu \leq \frac{1}{p} \int_{\mathbb{X}} |dw|_*^p \, d\nu - \frac1p \int_{\mathbb{X}} |du|_*^p \, d\nu,
\end{equation*}
which is the standard formulation for gradient flows on metric measure spaces in terms of {\it evolution variational inequalities}, see \cite{AGSBook}. { Alternatively, one can see this inequality as a consequence of the inclusion $u_t(t) \in \Delta_{p,\nu}(u(t))$.}

Finally, let us note that the characterisation of the solutions to the $p$-Laplace evolution equation introduced in this Section agrees with the notion of variational solutions, which goes back to the study of the gradient flow of the area functional by Lichnewsky and Temam in \cite{LT} and was formally introduced in \cite{BDM}.

\begin{corollary}
Suppose that $u$ is a weak solution of the Cauchy problem \eqref{eq:p-Lapcauchy}. Then, for any $v \in L^1_w(0,T; W^{1,p}(\mathbb{X},d,\nu))$ with $\partial_t v \in L^2(\mathbb{X} \times [0,T])$ and $v(0) \in L^2(\mathbb{X},\nu)$ we have
\begin{equation*}
\int_0^T \int_{\mathbb{X}} \partial_t v(v-u) \, d\nu \, dt + \frac1p \int_0^T \int_{\mathbb{X}} |dv(t)|_*^p \, d\nu - \frac1p \int_0^T \int_{\mathbb{X}} |du(t)|_*^p \, d\nu
\end{equation*}
\begin{equation}\label{varineqplaplace}
\geq \frac12 \Vert (v-u)(T) \Vert^2_{L^2(\mathbb{X},\nu)} - \frac12 \Vert v(0)-u_0 \Vert^2_{L^2(\mathbb{X},\nu)}.
\end{equation}
 Note that by \cite[Theorem 3.2]{Brezis} the weak solutions also have this regularity.
\end{corollary}

\begin{proof}
Given a test function $v$ as above, we want to show that \eqref{varineqplaplace} holds. We start by computing the term with the time derivative using the characterisation of weak solutions. Using the definition of the divergence, we have
$$\int_0^T \int_{\mathbb{X}} \partial_t u \, (v - u) \, d \nu \, dt = \int_0^T \int_{\mathbb{X}} \mbox{div}(X(t)) (v - u) \, d \nu \, dt $$
$$ = - \int_0^T \int_{\mathbb{X}} dv(t)(X(t)) \, d\nu \, dt + \int_0^T \int_{\mathbb{X}} du(t)(X(t)) \, d\nu \, dt.$$
Also,
$$\int_0^T \int_{\mathbb{X}} (\partial_t v - \partial_t u )(v - u) \, d \nu \, dt = \frac12 \Vert (v-u)(T) \Vert^2_{L^2(\mathbb{X},\nu)} - \frac12 \Vert v(0)-u_0 \Vert^2_{L^2(\mathbb{X},\nu)}. $$
Since $u$ is a weak solution, we add the two equalities and get
$$\int_0^T \int_{\mathbb{X}} \partial_t v \, (v - u) \, d \nu \, dt = - \int_0^T \int_{\mathbb{X}} dv(t)(X(t)) \, d\nu \, dt + \int_0^T \int_{\mathbb{X}} du(t)(X(t)) \, d\nu \, dt $$
$$ + \frac12 \Vert (v-u)(T) \Vert^2_{L^2(\mathbb{X},\nu)} - \frac12 \Vert v(0)-u_0 \Vert^2_{L^2(\mathbb{X},\nu)} \geq - \frac1p \int_0^T \int_{\mathbb{X}} \vert dv(t) \vert_*^p d\nu dt - \frac1q \int_0^T \int_{\mathbb{X}} \vert X(t) \vert^q d\nu dt $$
$$ + \int_0^T \int_{\mathbb{X}} \vert du(t) \vert_*^p d\nu dt + \frac12 \Vert (v-u)(T) \Vert^2_{L^2(\mathbb{X},\nu)} - \frac12 \Vert v(0)-u_0 \Vert^2_{L^2(\mathbb{X},\nu)}$$
$$ \geq - \frac1p \int_0^T \int_{\mathbb{X}} \vert dv(t) \vert_*^p d\nu dt + \frac1p \int_0^T \int_{\mathbb{X}} \vert du(t) \vert_*^p d\nu dt + \frac12 \Vert (v-u)(T) \Vert^2_{L^2(\mathbb{X},\nu)} - \frac12 \Vert v(0)-u_0 \Vert^2_{L^2(\mathbb{X},\nu)}.$$
\end{proof}

\section{Anzellotti pairings on metric measure spaces}\label{sec:Anzellotti}

In order to study the total variation flow, we need to introduce some additional assumptions on the metric measure space $(\mathbb{X},d,\nu)$. This is due to the fact solutions in order for the Cheeger energy $\mathsf{Ch}_1$ to be lower semicontinuous, it needs to be defined on the space of functions of bounded variation $BV(\mathbb{X},d,\nu)$ and not on the Sobolev space $W^{1,1}(\mathbb{X},d,\nu)$. Thus, we need to extend parts of the linear differential structure to the BV case, and to this end we need we will require that we can approximated BV functions with Lipschitz functions in a suitable way. Our strategy will be to introduce a metric version of the Anzellotti pairings introduced in \cite{Anz} and prove a Gauss-Green formula which will work as a replacement of the integration by parts given by the definition of the divergence of a vector field in $\mathcal{D}^{\infty,p}(\mathbb{X})$.

In this Section, we suppose that the metric space $(\mathbb{X},d)$ is complete, separable, equipped with a doubling measure $\nu$, and that the metric measure space $(\mathbb{X},d,\nu)$ supports a weak $(1,1)$-Poincar\'e inequality. In particular, these assumptions imply that $\mathbb{X}$ is locally compact, see \cite[Proposition 3.1]{BB}.

\subsection{An approximation result}

In order to define a generalised version of Anzellotti pairings, we will need to approximate a $BV$ function by regular enough functions, in the spirit of \cite[Lemma 5.2]{Anz}. Existence of a sequence of locally Lipschitz functions which approximate the desired function in the strict topology is automatic by virtue of Definition \ref{dfn:totalvariationonmetricspaces}; we will prove that we may require some additional properties of the approximating sequence.

\begin{lemma}\label{lem:lipschitzapproximation}

Suppose that $u \in BV(\mathbb{X},d, \nu)$. There exists a sequence of Lipschitz functions $u_n \in \mbox{Lip}(\mathbb{X}) \cap BV(\mathbb{X},d,\nu)$ such that: \\
$(1)$ $u_n \rightarrow u$ strictly in $BV(\mathbb{X},d,\nu)$;
\\
$(2)$ Let $p \in [1,\infty)$. If $u \in L^p(\mathbb{X},\nu)$, then $u_n \in L^p(\mathbb{X},\nu)$ and $u_n \rightarrow u$ in $L^p(\mathbb{X},\nu)$;
\\
$(3)$ If $u \in L^\infty(\mathbb{X},\nu)$, then $u_n \in L^\infty(\mathbb{X},\nu)$ and $u_n \rightharpoonup u$ weakly* in $L^\infty(\mathbb{X},\nu)$.
\end{lemma}

\begin{proof}
(1) Take any sequence of locally Lipschitz functions $u_n \in \mbox{Lip}_{loc}(\mathbb{X})$ convergent to $u$ in $L^1(\mathbb{X},\nu)$ given by Definition \ref{dfn:totalvariationonmetricspaces}; in particular, $u_n \in W^{1,1}(\mathbb{X},d,\nu)$. Since $(\mathbb{X},d,\nu)$ supports the weak $(1,1)$-Poincar\'e inequality, Lipschitz functions are dense in $W^{1,1}(\mathbb{X},d,\nu)$. Now, let $u_{n_k} \in \mbox{Lip}(\mathbb{X})$ be a sequence of Lipschitz functions which approximates $u_n$ in $W^{1,1}(\mathbb{X},d,\nu)$ as $k \rightarrow \infty$; by a diagonal argument we may choose a Lipschitz subsequence $u_{n_{k(n)}}$ which converges to $u$ strictly in $BV(\mathbb{X},d,\nu)$. For simplicity, we will denote this sequence by $u_n$.

(2) We need to show that after a suitable modification the functions $u_n$ constructed above additionally satisfy $u_n \in L^p(\mathbb{X},\nu)$ and $u_n \rightarrow u$ in $L^p(\mathbb{X},\nu)$. Given $v \in L^p(\mathbb{X},\nu)$, denote
\begin{equation}\label{eq:restrictionoffunction}
v_M(x) = \threepartdef{M}{v(x) > M;}{v(x)}{v(x) \in [-M,M];}{-M}{v(x) < -M.}
\end{equation}
Observe that $v_M \in L^\infty(\mathbb{X},\nu)$ and $v_M \rightarrow v$ in $L^p(\mathbb{X},\nu)$ as $M \rightarrow \infty$.

Now, recall that $u_n \rightarrow u$ strictly in $BV(\mathbb{X},d,\nu)$ as $n \rightarrow \infty$. Moreover, for every $M > 0$, we also have that $(u_n)_M \rightarrow u_M$ in $L^p(\mathbb{X},\nu)$ as $n \rightarrow \infty$:
$$ \| (u_n)_M - u_M \|_{L^p(\mathbb{X},\nu)}^p = \int_0^\infty t^{p-1} \, \nu(\{ |(u_n)_M - u_M| > t \}) \, dt \leq $$
$$ \leq (2M)^{p-1} \int_0^\infty \nu(\{ |(u_n)_M - u_M| > t \}) \, dt = (2M)^{p-1} \| (u_n)_M - u_M \|_{L^1(\mathbb{X},\nu)} \rightarrow 0,$$
because $|(u_n)_M - u_M| \leq 2M$ and $u_n \rightarrow u$ in $L^1(\mathbb{X},\nu)$. Hence, by a diagonalisation argument, there exists a sequence $(u_{n_k})_{M_k}$ such that
$$ (u_{n_k})_{M_{k}} \rightarrow u \quad \mbox{ in } L^p(\mathbb{X},\nu).$$
Moreover, this sequence also converges strictly in $BV(\mathbb{X},d,\nu)$, since truncations do not increase the slope:
$$ |Du|_\nu(\mathbb{X}) \leq \liminf_{k \rightarrow \infty} \int_{\mathbb{X}} |\nabla (u_{n_k})_{M_k}| \, d\nu \leq \liminf_{k \rightarrow \infty} \int_{\mathbb{X}} |\nabla u_{n_k}| \, d\nu = |Du|_\nu(\mathbb{X}).$$
Hence, possibly replacing the sequence $u_n$ by $(u_{n_k})_{M_k}$, we may require that $u_n \in L^p(\mathbb{X},\nu)$ and $u_n \rightarrow u$ strictly in $BV(\mathbb{X},d,\nu)$ and in the norm convergence in $L^p(\mathbb{X},\nu)$. Hence, up to the modification of $u_n$ described above, we may require that $u_n$ satisfies these properties as well.

(3) Notice that if $u_n \rightarrow u$ is the sequence given by  \eqref{dfn:totalvariationonmetricspaces}, then $(u_n)$ is bounded in $L^\infty(\mathbb{X},\nu)$ by $\| u \|_{L^\infty(\mathbb{X},\nu)}$ and by the argument from the proof of point (2) it converges to $u$ in $L^p(\mathbb{X},\nu)$ for every $p \in [1,\infty)$. Hence, it admits a weakly* convergent subsequence $u_{n_k}$. By the uniqueness of the weak* limit, we have $u_{n_k} \rightharpoonup u$ weakly* in $L^\infty(\mathbb{X},\nu)$.
\end{proof}

\subsection{Introducing the pairing}
	
The main goal of this subsection is to define the pairing $(X, Du)$ between a vector field $X$ with integrable divergence and a $BV$ function $u$. This will be a metric space analogue of the classic Anzellotti pairing introduced in \cite{Anz}.

Assume that  $X \in L^\infty(T\mathbb{X})$ and $u \in BV(\mathbb{X}, d, \nu)$. As in the case of classical Anzellotti pairings, we will additionally assume a joint regularity condition on $u$ and $X$ which makes the pairing well-defined. The condition is as follows: for $p \in [1,\infty)$, we have
\begin{equation}\label{Anzellotti:assumption}
\mbox{div}(X) \in L^p(\mathbb{X},\nu), \quad u \in BV(\mathbb{X},d, \nu) \cap L^{q}(\mathbb{X},\nu), \quad \frac{1}{p} + \frac{1}{q} = 1.
\end{equation}
 In other words, we have $X \in \mathcal{D}^{\infty,p}(\mathbb{X})$. In the proofs, we will sometimes differentiate between the cases when $p > 1$ and $p = 1$.

\begin{definition}
Suppose that the pair $(X, u)$ satisfies the condition \eqref{Anzellotti:assumption}. Then, given a Lipschitz function  $f \in \mbox{Lip}(\mathbb{X})$ with compact support, we set
$$ \langle (X, Du), f \rangle := -\int_{\mathbb{X}} u \, \mbox{div}(fX) \, d\nu =  -\int_{\mathbb{X}} u \, df(X) \, d\nu - \int_{\mathbb{X}} u f \mbox{div}(X) \, d\nu.$$
\end{definition}

\begin{proposition}\label{prop:boundonAnzellottipairing}
$(X, Du)$ is a Radon measure which is absolutely continuous with respect to $|Du|_\nu$.  Moreover, for every Borel set $A \subset \mathbb{X}$ we have
$$ \int_A |(X,Du)| \leq \| X \|_\infty \int_A |Du|_\nu.$$
\end{proposition}

\begin{proof}
For now, assume additionally that $u \in \mbox{Lip}(\mathbb{X})$. We note that  $fX \in \mathcal{D}^{\infty,p}(\mathbb{X})$ for all Lipschitz functions $f \in \mbox{Lip}(\mathbb{X})$ with compact support. Hence, by  the $L^\infty$-linearity of the differential, we have
$$ |\langle (X, Du), f \rangle| = \left|-\int_{\mathbb{X}} u \, \mbox{div}(fX) \, d\nu \right| =  \left|\int_{\mathbb{X}} du(fX) \, d\nu \right| = \left|\int_{\mathbb{X}} f \cdot du(X) \, d\nu \right| $$
$$ \leq \| f \|_\infty \left| \int_{\mathbb{X}} du(X) \, d\nu \right| \leq \| f \|_\infty \left|\int_{\mathbb{X}} |du|_* |X| \, d\nu \right| \leq \| f \|_\infty \| X \|_\infty \int_{\mathbb{X}} |Du| \, d\nu,$$
where in the last inequality we used that $|du|_* = |Du|$ $\nu$-a.e.

When $u \in BV(\mathbb{X},d,\nu)$ and it satisfies the assumption \eqref{Anzellotti:assumption}, take the sequence $u_n \in \mbox{Lip}(\mathbb{X})$ given by Lemma \ref{lem:lipschitzapproximation}. Notice that for any fixed $g \in \mbox{Lip}(\mathbb{X})$ with compact support we have
$$\langle (X, Du_j), g \rangle = -\int_{\mathbb{X}} u_j \, \mbox{div}(gX) \, d\nu \to - \int_{\mathbb{X}} u \, \mbox{div}(gX)  \, d\nu  = \langle (X,Du), g \rangle,$$
where the convergence is guaranteed by Lemma \ref{lem:lipschitzapproximation}. Indeed,  when the assumption \eqref{Anzellotti:assumption} is satisfied with $p > 1$, it works because $u_i \rightarrow u$ in $L^q(\mathbb{X},\nu)$ and $\mbox{div}(X) \in L^p(\mathbb{X},\nu)$ (so also $\mbox{div}(gX) \in L^p(\mathbb{X},\nu)$); on the other hand, when the assumption \eqref{Anzellotti:assumption} is satisfied with $p = 1$, it works because $u_i \rightarrow u$ weakly* in $L^\infty(\mathbb{X},\nu)$ and $\mbox{div}(X) \in L^1(\mathbb{X},\nu)$. Hence,
$$ |\langle (X, Du), f \rangle| = \lim_{n \rightarrow \infty} |\langle (X, Du_n), f \rangle| $$ $$\leq \lim_{n \rightarrow \infty} \| f \|_\infty \| X \|_\infty {\int_{\mathbb{X}} |Du_n| \, d\nu} = \| f \|_\infty \| X \|_\infty | Du |_\nu(\mathbb{X}).$$
Thus, $(X, Du)$ is a continuous functional on the space of Lipschitz functions. Since Lipschitz functions are dense in continuous functions, $(X,Du)$ defines a continuous functional on the space $C(\mathbb{X})$. Since by our assumptions $\mathbb{X}$ is locally compact, by Riesz representation theorem $(X,Du)$ is a Radon measure, $|(X, Du)| \ll \| X \|_\infty |Du|_\nu$ as measures and it satisfies the desired bound.
\end{proof}

Before we prove the Green's formula, we require one more technical result.
		
\begin{lemma}\label{lem:continuityofXDu}
Suppose that $u_i \rightarrow u$ as in the statement of Lemma \ref{lem:lipschitzapproximation}. Assume that the pair $(X,u)$ satisfies  the condition \eqref{Anzellotti:assumption}. Then
$$ \int_{\mathbb{X}} (X, Du_i) \rightarrow \int_{\mathbb{X}} (X,Du).$$
\end{lemma}

\begin{proof}
Fix $\varepsilon > 0$. Choose an open set $A \subset \subset \mathbb{X}$ such that
$$\int_{\mathbb{X} \backslash A} |Du|_\nu < \varepsilon.$$
Let $g \in \mbox{Lip}(\mathbb{X})$ be such that $0 \leq g \leq 1$ in $\mathbb{X}$, $g \equiv 1$ in $A$ and $g$ has compact support. We can choose such $g$ thanks to the Tietze extension theorem and density of Lipschitz functions in $C(\mathbb{X})$. We write $1 = g + (1 - g)$ and estimate
$$ \bigg| \int_{\mathbb{X}} (X, Du_j) - \int_{\mathbb{X}} (X, Du) \bigg| \leq \bigg| \langle (X, Du_j), g \rangle -  \langle (X, Du), g \rangle  \bigg| + $$
$$ + \int_{\mathbb{X}} |(X, Du_j)| (1-g) + \int_{\mathbb{X}} |(X, Du)| (1-g).$$
Arguing as in the proof of Proposition \ref{prop:boundonAnzellottipairing}, we see that for any fixed $g \in \mbox{Lip}(\mathbb{X})$ with compact support we have $\langle (X, Du_j), g \rangle \rightarrow \langle (X, Du), g \rangle$. Moreover, we have
$$\int_{\mathbb{X}} (1-g) |(X, Du)| \leq \int_{{\mathbb{X}} \backslash A} |(X, Du)| \leq \| X \|_\infty \int_{{\mathbb{X}} \backslash A} |Du|_\nu < \varepsilon \| X \|_\infty$$
and similarly
$$\limsup_{j \rightarrow \infty} \int_{\mathbb{X}} (1-g) |(X, Du_j)| \leq \limsup_{j \rightarrow \infty} \| X \|_\infty \int_{{\mathbb{X}} \backslash A} |Du_j|_\nu \leq \varepsilon \| X \|_\infty ,$$
so we can make the right hand side arbitrarily small.
\end{proof}

Now, we prove the result which motivates the construction of the Anzellotti pairings above. Namely, we show that the Green formula can be extended to the setting of BV functions in place of Lipschitz functions.
	
\begin{theorem}\label{thm:generalgreensformula}
Suppose that the pair $(X,u)$ satisfies  the condition \eqref{Anzellotti:assumption}. Then
$$ \int_\mathbb{X} u \,  \mbox{div}(X) \, d\nu + \int_\mathbb{X} (X,Du) = 0.$$
\end{theorem}
	
\begin{proof}
Using Lemma \ref{lem:lipschitzapproximation}, given $u \in BV(\mathbb{X},d,\nu)$ we find a sequence $u_i \in \mathrm{Lip}(\mathbb{X})$ such that $u_i \rightarrow u$ strictly.  By the definition of divergence, if we take $g = u_i$ we have
\begin{equation}\label{eq:greenbeforelimit}
\int_{\mathbb{X}} u_i \, \mathrm{div}(X) \, d\nu + \int_{\mathbb{X}} du_i(X) \, d\nu = 0.
\end{equation}
Notice that because $u_i$ are Lipschitz, we have
$$\int_{\mathbb{X}}  du_i(X) \, d\nu = \int_{\mathbb{X}} (X, Du_i).$$
To see this, let $g \in \mbox{Lip}({\mathbb{X}})$ have bounded support. Then, by the $L^\infty$-linearity of the differential, we have
$$ \int_{\mathbb{X}} g \, du_i(X) \, d\nu = \int_{\mathbb{X}} du_i(gX) = - \int_{\mathbb{X}} u_i \, \mbox{div}(gX) \, d\nu = \langle (X, Du_i), g \rangle.$$
Hence, integration with respect to $du_i(X) d\nu$ coincides with integration with respect $(X,Du_i)$ as a functional on Lipschitz functions with compact support, hence $du_i(X) d\nu$ and $(X,Du_i)$ coincide as measures.

Now, we pass to the limit $i \rightarrow \infty$ in equation \eqref{eq:greenbeforelimit}:
$$ 0 = \lim_{i \rightarrow \infty} \bigg( \int_{\mathbb{X}} u_i \, \mbox{div}(X) \, d\nu + \int_{\mathbb{X}} (X, Du_i) \bigg) =  \int_{\mathbb{X}} u \, \mbox{div}(X) \, d\nu + \int_{\mathbb{X}} (X, Du),$$
where in the first summand we use the assumption \eqref{Anzellotti:assumption} and in the second summand we use Lemma \ref{lem:continuityofXDu}. Hence, the Gauss-Green formula is proved.
\end{proof}

\subsection{Coarea formula for $(X,Du)$}

By Proposition \ref{prop:boundonAnzellottipairing}, the measure $(X,Du)$ is absolutely continuous with respect to the measure $|Du|_\nu$ and the estimate is uniform. Hence, by the Radon-Nikodym theorem, the measure $(X,Du)$ has a density $\theta(X,Du,x) \in L^\infty(\mathbb{X}, |Du|_\nu)$. Then,
\begin{equation}\label{Borel}
\int_B (X,Du) = \int_B \theta(X,Du,x) \, d|Du|_\nu, \quad \hbox{for all Borel sets} \ B \subset \mathbb{X}.
\end{equation}
Moreover, { $|\theta(X,Du,x)| \leq \| X \|_\infty$} $|Du|_\nu$-a.e. in $\mathbb{X}$. In this subsection, we study some properties of this density, and along the way we prove a co-area formula for the measure $(X,Du)$. We introduce the following notation: given $u \in BV(\mathbb{X},d,\nu)$, denote by $E_{u,t}$ the $t$-superlevel set of $u$, i.e. $E_{u,t} = \{ x \in \mathbb{X}: u(x) > t \}$.

\begin{theorem}\label{thm:pairingcoareaformula}
Suppose that the pair $(X,u)$ satisfies the condition \eqref{Anzellotti:assumption}. Then: \\
(1) For all $f \in \mbox{Lip}(\mathbb{X})$ with compact support, we have
\begin{equation*}
\langle (X,Du), f \rangle = \int_{-\infty}^\infty \langle (X,D\chi_{E_{u,t}}),f \rangle \, dt;
\end{equation*}
(2) $\theta(X,Du,x) = \theta(X,D\chi_{E_{u,t}},x)$ $|D\chi_{E_{u,t}}|_\nu$-a.e. in $\mathbb{X}$ for $\mathcal{L}^1$-a.e. $t \in \mathbb{R}$; \\
(3) For all Borel sets $B \subset \mathbb{X}$, we have the following co-area formula
\begin{equation}
\int_B (X,Du) = \int_{-\infty}^\infty \bigg( \int_B (X,D\chi_{E_{u,t}}) \bigg) \, dt.
\end{equation}
\end{theorem}

\begin{proof}
(1) First, suppose that $u \in BV(\mathbb{X},d,\nu) \cap L^\infty(\mathbb{X},\nu)$. Denote $M = \| u \|_{L^\infty(\mathbb{X},\nu)}$ and $K = \mbox{supp}(f)$. Furthermore, denote by $u^+$ and $u^-$ the positive and negative parts of $u$. Then,
\begin{equation*}
\langle (X,Du), f \rangle = - \int_{\mathbb{X}} u \, \mbox{div}(fX) \, d\nu = - \int_{\mathbb{X}} u^+ \, \mbox{div}(fX) \, d\nu + \int_{\mathbb{X}} u^- \, \mbox{div}(fX) \, d\nu
\end{equation*}
\begin{equation*}
= - \int_{K} u^+ \, \mbox{div}(fX) \, d\nu + \int_{K} u^- \, \mbox{div}(fX) \, d\nu = - \int_{K} \int_0^M \chi_{E_{u,t}} \, \mbox{div}(fX) \, dt \, d\nu
\end{equation*}
\begin{equation*}
+ \int_{K} \int_{-M}^0 (1 - \chi_{E_{u,t}}) \, \mbox{div}(fX) \, dt \, d\nu = - \int_{K} \int_{-M}^M \chi_{E_{u,t}} \, \mbox{div}(fX) \, dt \, d\nu + \int_K \int_{-M}^0 \mbox{div}(fX) \, dt \, d\nu
\end{equation*}
\begin{equation*}
= - \int_{\mathbb{X}} \int_{-M}^M \chi_{E_{u,t}} \, \mbox{div}(fX) \, dt \, d\nu - \int_{\mathbb{X}} M \mbox{div}(fX) \, dt \, d\nu = - \int_{-M}^M \langle (X,D\chi_{E_{u,t}}), f \rangle \, dt,
\end{equation*}
where the last equality follows from the fact that integrating by parts $M \mbox{div}(fX)$ gives zero since $M$ is a constant.

Now, suppose that $u \notin L^\infty(\mathbb{X},\nu)$. For $k > 0$, denote by $T_k(u)$ the truncation of $u$ at level $k$, i.e.
\begin{equation*}
T_k(u) = \threepartdef{k}{u(x) \geq k;}{u(x)}{u(x) \in (-k,k);}{-k}{u(x) \leq -k.}
\end{equation*}
Then, by the standard co-area formula (see \cite{Miranda1}) $T_k(u) \in BV(\mathbb{X},d,\nu)$ and $\int_{\mathbb{X}} |DT_k(u)|_\nu \leq \int_{\mathbb{X}} |Du|_\nu$. Hence, the sequence $T_k(u)$ converges strictly to $u$ as $k \rightarrow \infty$ (and in $L^p(\mathbb{X},\nu)$ if $u \in L^p(\mathbb{X},\nu)$). Since $T_k(u) \in BV(\mathbb{X},d,\nu) \cap L^\infty(\mathbb{X},\nu)$, we have
\begin{equation}\label{eq:coareaforpairingtruncation}
\langle (X,DT_k(u)), f \rangle = - \int_{-k}^k \langle (X,D\chi_{E_{T_k(u),t}}), f \rangle \, dt.
\end{equation}
Now, the left hand side converges to $\langle (X,Du), f \rangle$ as in the proof of Proposition \ref{prop:boundonAnzellottipairing}. On the right hand side, notice that since the bound
\begin{equation*}
\bigg| \int_{-k}^k \langle (X,D\chi_{E_{T_k(u),t}}), f \rangle \, dt \bigg| = | \langle (X,DT_k(u)), f \rangle | \leq \| f \|_\infty \| X \|_\infty |Du|_\nu(\mathbb{X})
\end{equation*}
does not depend on $k$, it also holds in the limit $k \rightarrow \infty$. Now, notice that
\begin{equation*}
\bigg| \int_{-k}^k \langle (X,D\chi_{E_{T_k(u),t}}), f \rangle \, dt - \int_{-\infty}^\infty \langle (X,D\chi_{E_{u,t}}), f \rangle \, dt \bigg|
\end{equation*}
\begin{equation*}
\leq \| f \|_\infty \| X \|_\infty \bigg( \int_{-\infty}^k \int_{\mathbb{X}} |D\chi_{E_{u,t}}|_\nu \, dt + \int_{k}^\infty \int_{\mathbb{X}} |D\chi_{E_{u,t}}|_\nu \, dt \bigg),
\end{equation*}
which goes to zero as $k \rightarrow \infty$ by the co-area formula. Hence, we may pass to the limit also in the right hand side of \eqref{eq:coareaforpairingtruncation}, which proves part (1) of the Theorem.

(2) For $a,b \in \mathbb{R}$ with $a < b$, denote by $T_{a,b}(u)$ the truncation of $u$ at levels $a,b$, i.e.
\begin{equation*}
T_{a,b}(u) = \threepartdef{b}{u(x) \geq b;}{u(x)}{u(x) \in (a,b);}{a}{u(x) \leq a.}
\end{equation*}
Then, by the standard co-area formula $T_{a,b}(u) \in BV(\mathbb{X},d,\nu)$ and $\int_{\mathbb{X}} |DT_{a,b}(u)|_\nu \leq \int_{\mathbb{X}} |Du|_\nu$. Let us see that
\begin{equation}\label{eq:equalitythetatruncation}
\theta(X,Du,x) = \theta(X,DT_{a,b}(u),x) \qquad |DT_{a,b}(u)|_\nu-\mbox{a.e. in } \mathbb{X}.
\end{equation}
Suppose otherwise. Then, there exists a Borel set $B \subset \mathbb{X}$ such that $u(x) \in [a,b]$ $\nu$-a.e. on $B$ and $\theta(X,Du,x) > \theta(X,DT_{a,b}(u),x)$ $|DT_{a,b}(u)|_\nu$-a.e. on $B$ (or with the opposite inequality, that case is handled similarly). Hence,
\begin{equation*}
\int_B (X,Du) = \int_B \theta(X,Du,x) |Du|_\nu = \int_B \theta(X,Du,x) |DT_{a,b}(u)|_\nu
\end{equation*}
\begin{equation}\label{eq:truncationbycontradiction}
> \int_B \theta(X,DT_{a,b}(u),x) |DT_{a,b}(u)|_\nu = \int_B (X, DT_{a,b}(u)).
\end{equation}
Now, notice that
\begin{equation*}
\bigg| \int_B (X,Du) - \int_B (X, DT_{a,b}(u)) \bigg| = \bigg| \int_B (X, D(u - T_{a,b}(u))) \bigg| \leq \| X \|_\infty \int_B |D(u-T_{a,b}(u))|_\nu
\end{equation*}
\begin{equation*}
= \int_{-\infty}^\infty \int_B |D\chi_{E_{u-T_{a,b}(u),t}}|_\nu \, dt = \int_{-\infty}^a \int_B |D\chi_{E_{u,t}}|_\nu \, dt + \int_{b}^\infty \int_B |D\chi_{E_{u,t}}|_\nu \, dt = 0,
\end{equation*}
since $a \leq u \leq b$ $\nu$-a.e. on $B$. This gives a contradiction with \eqref{eq:truncationbycontradiction}, so \eqref{eq:equalitythetatruncation} holds.

Now, recall that by part (1) of the Theorem for all $f \in \mbox{Lip}(\mathbb{X})$ with compact support we have
\begin{equation*}
\langle (X,Du), f \rangle = \int_{-\infty}^\infty \langle (X,D\chi_{E_{u,t}}),f \rangle \, dt.
\end{equation*}
Hence, for any $a,b \in \mathbb{R}$ with $a < b$ we have
\begin{equation*}
\int_{\mathbb{X}} \theta(X,DT_{a,b}(u),x) \, f(x) \, |DT_{a,b}(u)|_\nu = \int_a^b \bigg( \int_{\mathbb{X}} \theta(X,D\chi_{E_{T_{a,b}(u),t}},x) \, f(x) \, |D\chi_{E_{T_{a,b}(u),t}}|_\nu \bigg) \, dt.
\end{equation*}
Now, we use property \eqref{eq:equalitythetatruncation} (on the left hand side) the fact that $u$ and $T_{a,b}(u)$ coincide on the domain of integration (on both sides). Then
\begin{equation*}
\int_{\mathbb{X}} \theta(X,Du,x) \, f(x) \, |DT_{a,b}(u)|_\nu = \int_a^b \bigg( \int_{\mathbb{X}} \theta(X,D\chi_{E_{u,t}},x) \, f(x) \, |D\chi_{E_{u,t}}|_\nu \bigg) \, dt.
\end{equation*}
Now, by the standard co-area formula (on the left hand side), we have
\begin{equation*}
\int_a^b \bigg( \int_{\mathbb{X}} \theta(X,Du,x) \, f(x) \, |D\chi_{E_{u,t}}|_\nu \bigg) \, dt = \int_a^b \bigg( \int_{\mathbb{X}} \theta(X,D\chi_{E_{u,t}},x) \, f(x) \, |D\chi_{E_{u,t}}|_\nu \bigg) \, dt.
\end{equation*}
Since $a$ and $b$ were arbitrary, the equality under the integral needs to hold for $\mathcal{L}^1$-a.e. $t \in \mathbb{R}$, so
\begin{equation*}
\int_{\mathbb{X}} \theta(X,Du,x) \, f(x) \, |D\chi_{E_{u,t}}|_\nu = \int_{\mathbb{X}} \theta(X,D\chi_{E_{u,t}},x) \, f(x) \, |D\chi_{E_{u,t}}|_\nu
\end{equation*}
for $\mathcal{L}^1$-a.e. $t \in \mathbb{R}$. But then, since $f$ was arbitrary, by a density argument we get that $\theta(X,Du,x) = \theta(X,D\chi_{E_{u,t}},x)$ $|D\chi_{E_{u,t}}|_\nu$-a.e. in $\mathbb{X}$, which proves part (2) of the Theorem.

(3) By part (2) of the Theorem and the standard co-area formula, we have
\begin{equation*}
\int_B (X,Du) = \int_B \theta(X,Du,x) |Du|_\nu = \int_{-\infty}^\infty \bigg( \int_{B} \theta(X,Du,x) \, |D\chi_{E_{u,t}}|_\nu \bigg) \, dt
\end{equation*}
\begin{equation*}
= \int_{-\infty}^\infty \bigg( \int_{B} \theta(X,D\chi_{E_{u,t}},x) \, |D\chi_{E_{u,t}}|_\nu \bigg) \, dt = \int_{-\infty}^\infty \bigg( \int_{B} (X,D\chi_{E_{u,t}}) \bigg) \, dt,
\end{equation*}
so the Theorem is proved.
\end{proof}

\begin{corollary}\label{perfect}
Suppose that the pair $(X,u)$ satisfies the condition \eqref{Anzellotti:assumption}. If $T: \mathbb{R} \rightarrow \mathbb{R}$ is a Lipschitz continuous increasing function, then
\begin{equation*}
\theta(X,D(T \circ u),x) = \theta(X,Du,x) \qquad |Du|_\nu-\mbox{a.e. in } \mathbb{X}.
\end{equation*}
\end{corollary}

\begin{proof}
Notice that
\begin{equation*}
E_{u,t} = \{ x \in \mathbb{X}: \, u(x) > t \} = \{ x \in \mathbb{X}: \, (T \circ u)(x) > T(t) \} = E_{T \circ u, T(t)}.
\end{equation*}
Hence, by Theorem \ref{thm:pairingcoareaformula}, for $\mathcal{L}^1$-almost all $t \in \mathbb{R}$
\begin{equation*}
\theta(X,Du,x) = \theta(X,D\chi_{E_{u,t}},x) = \theta(X,D\chi_{E_{T \circ u, T(t)}},x) = \theta(X, D(T \circ u),x)
\end{equation*}
$|D\chi_{E_{u,t}}|$-a.e. in $\mathbb{X}$. Hence, this equality also holds $|Du|_\nu$-a.e.
\end{proof}

\section{The total variation flow}\label{sec:TVflow}

In this section we study  the Cauchy problem
\begin{equation}\label{Cauchy100}
\left\{ \begin{array}{ll} u_t (t,x) = {\rm div} \left(\frac{ D u (t,x)}{\vert Du(t,x) \vert_\nu } \right)  \quad &\hbox{in} \ \ (0, T) \times \mathbb{X},  \\[5pt] u(0,x) = u_0(x) \quad & \hbox{in} \ \  \mathbb{X}. \end{array} \right.
\end{equation}

In order to use the Anzellotti pairings and the Green formula introduced in the previous Section, we again suppose that the metric space $(\mathbb{X},d)$ is complete, separable, equipped with a doubling measure $\nu$, and that the metric measure space $(\mathbb{X},d,\nu)$ supports a weak $(1,1)$-Poincar\'e inequality. Consider the energy functional $\mathcal{TV} : L^2(\mathbb{X}, \nu) \rightarrow [0, + \infty]$ defined by
\begin{equation}\label{CauchyF1}
\mathcal{TV}(u):= \left\{ \begin{array}{ll} \vert Du \vert_\nu (\mathbb{X})  \quad &\hbox{if} \ u \in BV(\mathbb{X}, d, \nu) \cap L^2(\mathbb{X},\nu), \\ \\ + \infty \quad &\hbox{if} \ u \in  L^2(\mathbb{X}, \nu) \setminus BV(\mathbb{X}, d, \nu).\end{array}\right.
\end{equation}

 We also denote $\mathcal{TV}$ as $\mathsf{Ch}_1$. We have that $\mathcal{TV}$ is convex and lower semi-continuous with respect to the $L^2(\mathbb{X},\nu)$-convergence. Then, by the theory of maximal monotone operators (see \cite{Brezis}) there is a unique strong solution of the abstract Cauchy problem
\begin{equation}\label{ACP12}
\left\{ \begin{array}{ll} u'(t) + \partial \mathcal{TV}(u(t)) \ni 0, \quad t \in [0,T] \\[5pt] u(0) = u_0. \end{array}\right.
\end{equation}

Working as in the proof of Theorem \ref{thm:plaplaceflow}, but using Green formula (Theorem \ref{thm:generalgreensformula}) instead of the definition of the divergence, we obtain the following charaterisation of the subdifferential of $\mathcal{TV}$ operator.

\begin{definition}{\rm $(u,v) \in \mathcal{A}_1$ if and only if $u, v \in L^2(\mathbb{X}, \nu)$, $u \in BV(\mathbb{X}, d, \nu)$ and there exists a vector field  $X \in \mathcal{D}^{\infty,2}(\mathbb{X})$ with $\| X \|_\infty \leq 1$ such that the following conditions hold:
$$ -\mbox{div}(X) = v \quad \hbox{in } \mathbb{X}; $$
$$ (X, Du) = |Du|_\nu \quad \hbox{as measures}.$$
}
\end{definition}

To get the characterization of $\partial \mathcal{TV}$, we use again the Fenchel-Rockafellar duality theorem.

\begin{theorem}\label{thm:totalvariationflow}
$\partial \mathcal{TV} = \mathcal{A}_1$. Furthermore, the operator $\mathcal{A}_1$ is completely accretive and the domain of $\mathcal{A}_1$ is dense in $L^2(\mathbb{X}, \nu)$.
\end{theorem}

\begin{proof}
First, let us see that $\mathcal{A}_1 \subset \partial \mathcal{TV}$. Let $(u,v) \in \mathcal{A}_1$. Then, given $w \in L^2(\mathbb{X}, \nu) \cap BV(\mathbb{X}, d,\nu)$, we have
$$ \int_{\mathbb{X}} v(w-u) \, d \nu = - \int_{\mathbb{X}}  \mbox{div}(X)(w -u) \, d\nu = \int_{\mathbb{X}} (X, D(w-u))  = \int_{\mathbb{X}} (X, Dw)- \int_{\mathbb{X}} \vert Du \vert_\nu   $$ $$\leq \mathcal{TV}(w) - \mathcal{TV}(u),$$
and consequently, $(u,v) \in \partial \mathcal{TV}$. { Notice that since $\mathcal{A}_1 \subset \partial \mathcal{TV}$, the operator $\mathcal{A}_1$ is monotone.}

Now, the operator $\partial \mathcal{TV}$ is maximal monotone. Then, if we show that  $\mathcal{A}_1$ satisfies the range condition, by Minty Theorem we would also have that the operator $\mathcal{A}_1$ is maximal monotone, and consequently $\partial \mathcal{TV}= \mathcal{A}_1$. In order to finish the proof, let us see that $\mathcal{A}_1$ satisfies the range condition, i.e.
\begin{equation}\label{RCondN}
\hbox{Given} \ g \in  L^2(\mathbb{X}, \nu), \ \exists \, u \in D(\mathcal{A}_1) \ s.t. \ \  g \in u + \mathcal{A}_1(u).
\end{equation}
Now,
$$ g \in u + \mathcal{A}_1(u) \iff (u, g-u) \in \mathcal{A}_1,$$
so we need to show that  there exists a vector field  $X \in \mathcal{D}^{\infty,2}(\mathbb{X})$ with $\| X \|_\infty \leq 1$ such that the following conditions hold:
\begin{equation}\label{1RCondN} -\mbox{div}(X) = g-u \quad \hbox{in } \ \mathbb{X};
\end{equation}
\begin{equation}\label{2RCondN}(X, Du) = |Du|_\nu \quad \hbox{as measures}.
\end{equation}
We are going to prove \eqref{RCondN}  by means of the Fenchel-Rockafellar duality theorem. We set $U = W^{1,1}(\mathbb{X},d,\nu) \cap L^2(\mathbb{X},\nu)$, $V = L^1(T^{*} \mathbb{X})$, and the operator $A: U \rightarrow V$ is defined by the formula
$$
A(u) = du,
$$
where $du$ is the differential of $u$ in the sense of Definition \ref{dfn:differential}. Hence, $A$ is a linear and continuous operator. Moreover, the dual spaces to $U$ and $V$ are
$$
U^* = (W^{1,1}(\mathbb{X},d,\nu) \cap L^2(\mathbb{X},\nu))^*, \qquad V^* = L^\infty(T\mathbb{X}).
$$
We set $E: L^1(T^{*} \mathbb{X}) \rightarrow \mathbb{R}$ by the formula
\begin{equation}\label{Ieq:definitionofEN}
E(v) = \int_{\mathbb{X}} |v|_* \, d\nu.
\end{equation}
It is clear that the functional $E^*: L^\infty(T\mathbb{X}) \rightarrow [0,\infty]$ is given by the formula
\begin{equation}
E^*(v^*) = \| v^* \|_{L^\infty(T\mathbb{X})}.
\end{equation}
We also set $G:W^{1,1}(\mathbb{X},d,\nu) \cap L^2(\mathbb{X},\nu) \rightarrow \mathbb{R}$ by
$$G(u):= \frac12 \int_{\mathbb{X}} u^2 \, d\nu - \int_{\mathbb{X}} ug \, d\nu.$$
The functional $G^* : (W^{1,1}({\mathbb{X}},d,\nu) \cap L^2({\mathbb{X}},\nu))^* \rightarrow [0,+\infty ]$ is given by
$$G^*(u^*) = \displaystyle\frac12 \int_{\mathbb{X}} (u^*  + g)^2 \, d\nu.$$
Now, for  fixed $v^* \in L^\infty(T\mathbb{X})$ in the domain of $A^*$ and any $u \in W^{1,1}(\mathbb{X},d,\nu) \cap L^2(\mathbb{X},\nu)$, we have
$$ \int_{\mathbb{X}} u \, (A^*v^*) \, d\nu = \langle u, A^* v^* \rangle  = \langle v^*, Au \rangle = \int_{\mathbb{X}} du(v^*) \, d\nu,$$
so the definition of the divergence of $v^*$ is satisfied with
\begin{equation}\label{div0N}
\mbox{div}(v^*) = - A^* v^*.
\end{equation}
In particular, $\mbox{div}(v^*) \in L^2(\mathbb{X},\nu)$. In other words, the domain of $A^*$ is $\mathcal{D}^{\infty,2}(\mathbb{X})$.

Consider the energy functional $\mathcal{G}_1 : L^2(\mathbb{X}, \nu) \rightarrow (-\infty, + \infty]$ defined by
\begin{equation}
\mathcal{G}_1(u):= \mathsf{Ch}_1(u) + G(u).
\end{equation}
Notice that we may write it as
\begin{equation}\label{eq:primaltvflow}
\min_{u \in L^2(\mathbb{X}, \nu)} \mathcal{G}_1(u) = \inf_{u \in BV(\mathbb{X},d,\nu) \cap L^2(\mathbb{X},\nu)} \bigg\{ E(Au) + G(u) \bigg\}.
\end{equation}
Hence, its dual problem is
\begin{equation}\label{eq:dualtvflow}
\sup_{v^* \in L^\infty(T \mathbb{X})} \bigg\{  - E^*(v^*) - G^*(A^* v^*) \bigg\}.
\end{equation}
For $u_0 \equiv 0$ we have $E(Au_0) = 0 < \infty$, $G(u_0) = 0 < \infty$ and $E$ is continuous at $0$. Then, by the Fenchel-Rockafellar duality theorem, we have
\begin{equation}\label{eq:dualitygaptvflow}
\inf \eqref{eq:primaltvflow} = \sup \eqref{eq:dualtvflow}
\end{equation}
and
\begin{equation}\label{eq:dualexistencetvflow}
\hbox{the dual problem \eqref{eq:dualtvflow} admits at least one solution.}
\end{equation}
The functional $\mathsf{Ch}_1$ does not always have a minimiser in $W^{1,1}(\mathbb{X},d,\nu) \cap L^2(\mathbb{X},\nu)$. However, if we consider its relaxation
\begin{equation}\label{11fuctmeN}
\overline{\mathcal{G}_1}(u):= \mathcal{TV}(u) + G(u),
\end{equation}
then $\overline{\mathcal{G}_1}$ is coercive, convex and lower semi-continuous, so the minimization problem
$$\min_{u \in L^2(\mathbb{X}, \nu)} \overline{\mathcal{G}_1}(u)$$ admits an optimal solution $\overline{u} \in BV(\mathbb{X},d,\nu)$.

Now, let us take a sequence $u_n \in W^{1,1}(\mathbb{X},d,\nu)$ which converges strictly to $\overline{u}$ and also $u_n \to \overline{u}$ in $L^2(\mathbb{X}, \nu)$ (given by Lemma \ref{lem:lipschitzapproximation}). Then, it is a minimising sequence in \eqref{eq:primaltvflow}. Since we have \eqref{eq:dualitygaptvflow} and \eqref{eq:dualexistencetvflow}, we may use the $\varepsilon-$subdifferentiability property of minimising sequences, see \cite[Proposition V.1.2]{EkelandTemam}: for any minimising sequence $u_n$ for \eqref{eq:primaltvflow} and a maximiser $\overline{v}^*$ of \eqref{eq:dualtvflow}, we have
\begin{equation}\label{eq:epsilonsubdiff11N}
0 \leq E(Au_n) + E^*(-\overline{v}^*) - \langle -\overline{v}^*, Au_n \rangle \leq \varepsilon_n
\end{equation}
\begin{equation}\label{eq:epsilonsubdiff21N}
0 \leq G(u_n) + G^*(A^* \overline{v}^*) - \langle u_n, A^* \overline{v}^* \rangle \leq \varepsilon_n
\end{equation}
with $\varepsilon_n \rightarrow 0$.

From the second condition, using the definition of the convex conjugate, we have that for any $w \in L^2(\mathbb{X},\nu)$
\begin{equation*}
G^*(A^* \overline{v}^*) \geq \langle w, A^* \overline{v}^* \rangle - G(w),
\end{equation*}
so by equation \eqref{eq:epsilonsubdiff21N} we have
\begin{equation*}
G(w) - G(u_n) \geq \langle w, A^* \overline{v}^* \rangle - G^*(A^* \overline{v}^*) - G(u_n) \geq \langle w - u_n, A^* \overline{v}^* \rangle - \varepsilon_n,
\end{equation*}
so
$$G(w) - G(\overline{u}) \geq \langle (w - \overline{u}), A^* \overline{v}^* \rangle$$
and consequently $A^* \overline{v}^* \in \partial G(\overline{u}) = \{ \overline{u} - g \}$. Hence,
\begin{equation}\label{div1N}
-\mathrm{div}(\overline{v}^*) = \overline{u} - g.
\end{equation}
On the other hand, equation \eqref{eq:epsilonsubdiff11N} gives
\begin{equation}\label{eq:goodN}
0 \leq \bigg( \int_{\mathbb{X}} |du_n|_* \, d\nu + \int_{\mathbb{X}} du_n(\overline{v}^*) \, d\nu \bigg) \leq \varepsilon_n.
\end{equation}
Since $\int_{\mathbb{X}} |du_n|_* \, d\nu = \int_{\mathbb{X}} |Du_n|_\nu$, we have
\begin{equation}\label{inerior1N}
0 \leq  \int_{\mathbb{X}} \bigg( |Du_n|_\nu + du_n(\overline{v}^*) \bigg) d\nu \leq \varepsilon_n.
\end{equation}
Finally, keeping in mind that $-\mbox{div}(\overline{v}^*) = \overline{u} - g$, by Green's formula (Theorem \ref{thm:generalgreensformula}) we get
$$ \int_{\mathbb{X}} du_n(\overline{v}^*) \, d\nu = - \int_{\mathbb{X}} u_n \, \mbox{div}(\overline{v}^*) \, d\nu  =\int_{\mathbb{X}} u_n \, (\overline{u} - g) \, d\nu. $$
Then, applying again Green's formula, we have
$$\lim_{n \to \infty} \int_{\mathbb{X}} du_n(\overline{v}^*) \, d\nu = \int_{\mathbb{X}} \overline{u} \, (\overline{u}-g) \, d\nu  = -  \int_{\mathbb{X}} \overline{u} \, \mbox{div}(\overline{v}^*) \, d\nu  = \int_{\mathbb{X}} (\overline{v}^*, D\overline{u}). $$
Hence, since $u_n$ converges strictly to $\overline{u}$, having in mind \eqref{inerior1N} we get
$$ \int_{\mathbb{X}} |D\overline{u}|_\nu - \int_{\mathbb{X}} (-\overline{v}^*, D\overline{u}) = \lim_{n \rightarrow \infty} \bigg( \int_{\mathbb{X}} |Du_n|_\nu - \int_{\mathbb{X}} (-\overline{v}^*, D\overline{u}) \bigg) = 0.$$
This together with Proposition \ref{prop:boundonAnzellottipairing} implies that
\begin{equation}\label{measure1N}
(-\overline{v}^*, D\overline{u}) = |D\overline{u}|_\nu \quad \mbox{as measures in } {\mathbb{X}}.
\end{equation}
Hence, we have that the pair $(\overline{u}, -\overline{v}^*)$ satisfies \eqref{1RCondN} and \eqref{2RCondN}, therefore \eqref{RCondN} holds. Therefore, the operator $\mathcal{A}_1$ satisfies the range condition, so it is maximal monotone, which in turn implies $\mathcal{A}_1 = \partial \mathcal{TV}$.


Let $P_{0}$ denote the set of all functions $T\in C^{\infty}(\R)$ satisfying $0\le T'\le 1$ such that $T'$  is compactly supported, and $x=0$ is not contained in the support $\textrm{supp}(T)$ of $T$. To prove that $\mathcal{A}_1$ is  a completely accretive operator we must show that  (see \cite{ACMBook,BCr2})
$$ \int_{\mathbb{X}}T(u_1-u_2)(v_1-v_2)\, d\nu \geq 0 $$
for every $T\in P_{0}$ and every $(u_i,v_i) \in \mathcal{A}_1$, $i=1,2$. Recall that if $(u_i, v_i) \in \mathcal{A}_1$, $i =1,2$, we have $u_i \in BV(\mathbb{X}, d, \nu)$ and there exists  vector fields  $X_i \in \mathcal{D}^{\infty,2}(\mathbb{X})$ with $\| X_i \|_\infty \leq 1$ such that the following conditions hold:
\begin{equation}\label{e1caplaplaceBV}-\mbox{div}(X_i) = v_i \quad \hbox{in } \mathbb{X};
\end{equation}
\begin{equation}\label{e2caplaplaceBV}
 (X_i, Du_i) = |Du_i|_\nu \quad \hbox{as measures}.
\end{equation}
Making a similar computation as in the proof that $\mathcal{A}_1$ is monotone, notice that for every Borel set $B \subset \mathbb{X}$ we have
$$\int_{B} (X_1 - X_2, Du_1 - Du_2) =  \int_{B} \vert Du_1 \vert_\nu - \int_{B} (X_1, Du_2) +  \int_{B} \vert Du_2 \vert_\nu - \int_{B} (X_2, Du_1) \geq 0. $$

Hence, by \eqref{Borel},
$$ \int_{B} \theta(X_1 - X_2, D (u_1 - u_2), x ) \, d \vert D(u_1 - u_2) \vert_\nu = \int_{B} (X_1 - X_2, D (u_1 - u_2)) \geq 0$$
for all Borel set $B \subset \mathbb{X}$. Thus
$$\theta(X_1 - X_2, D (u_1 - u_2), x ) \geq 0, \quad \vert D(u_1 - u_2) \vert_\nu-\hbox{a.e. on} \ \mathbb{X}.$$
Moreover, since $|DT(u_1 - u_2)|_\nu$ is absolutely continuous with respect to $|D(u_1-u_2)|_\nu$, we also have
$$\theta(X_1 - X_2, D (u_1 - u_2), x ) \geq 0, \quad \vert DT(u_1 - u_2) \vert_\nu-\hbox{a.e. on} \ \mathbb{X}.$$
Then, applying the Green formula (Theorem \ref{thm:generalgreensformula}) and having in mind Corollary \ref{perfect}, we have
$$ \int_{\mathbb{X}}T(u_1-u_2)(v_1-v_2)\, d\nu = - \int_{\mathbb{X}} T(u_1 - u_2) (\mbox{div}(X_1)- \mbox{div}(X_2)) \, d\nu $$
$$=\int_{\mathbb{X}} (X_1 - X_2, D T(u_1 - u_2)) =  \int_{\mathbb{X}} \theta(X_1 - X_2, D T(u_1 - u_2), x ) \, d \vert DT(u_1 - u_2) \vert_\nu $$
$$= \int_{\mathbb{X}}\theta(X_1 - X_2, D(u_1 - u_2), x ) \, d \vert  DT(u_1 - u_2) \vert_\nu \geq 0,$$
 so $\mathcal{A}_1$ is completely accretive.

Finally, by \cite[Proposition 2.11]{Brezis}, we have
$$ D(\partial \mathcal{TV}) \subset  D(\mathcal{TV}) =  BV(\mathbb{X},d,\nu) \cap L^2(\mathbb{X},\nu) \subset \overline{D(\mathcal{TV})}^{L^2(\mathbb{X}, \nu)} \subset \overline{D(\partial \mathcal{TV})}^{L^2(\mathbb{X}, \nu)},$$
from which follows the density of the domain.
\end{proof}

As for the $p$-Laplace equation, we may give a more detailed characterisation of solutions in terms of variational inequalities. We present the equivalent characterisations in the following Corollary.

\begin{corollary}\label{cor:subdifferentialonelaplace}
The following conditions are equivalent: \\
$(a)$ $(u,v) \in \partial\mathcal{TV}$; \\
$(b)$ $(u,v) \in \mathcal{A}_1$, i.e. $u, v \in L^2(\mathbb{X}, \nu)$, $u \in BV(\mathbb{X}, d, \nu)$ and there exists a vector field  $X \in \mathcal{D}^{\infty,2}(\mathbb{X})$ with  $\| X \|_\infty \leq 1$ such that $-\mbox{div}(X) = v$ in $\mathbb{X}$ and
\begin{equation}
(X,Du) = |Du|_\nu  \quad \hbox{as measures in } \mathbb{X};
\end{equation}
$(c)$ $u, v \in L^2(\mathbb{X}, \nu)$, $u \in BV(\mathbb{X}, d, \nu)$ and there exists a vector field $X \in \mathcal{D}^{\infty,2}(\mathbb{X})$ with  $\| X \|_\infty \leq 1$ such that $-\mbox{div}(X) = v$ in $\mathbb{X}$ and for every $w \in L^2(\mathbb{X},\nu) \cap BV(\mathbb{X},d,\nu)$
\begin{equation}\label{eq:variationalinequalitytvflow}
\int_{\mathbb{X}} v(w-u) \, d\nu \leq \int_{\mathbb{X}} (X,Dw) - \int_{\mathbb{X}} |Du|_\nu;
\end{equation}
$(d)$ $u, v \in L^2(\mathbb{X}, \nu)$, $u \in BV(\mathbb{X}, d, \nu)$ and there exists a vector field  $X \in \mathcal{D}^{\infty,2}(\mathbb{X})$ with  $\| X \|_\infty \leq 1$ such that $-\mbox{div}(X) = v$ in $\mathbb{X}$ and for every $w \in L^2(\mathbb{X},\nu) \cap BV(\mathbb{X},d,\nu)$
\begin{equation}
\int_{\mathbb{X}} v(w-u) \, d\nu = \int_{\mathbb{X}} (X,Dw) - \int_{\mathbb{X}} |Du|_\nu.
\end{equation}
\end{corollary}

\begin{proof}
The equivalence of $(a)$ and $(b)$ is exactly the content of Theorem \ref{thm:totalvariationflow}. To see that $(b)$ implies $(d)$, multiply the equation $v = -\mbox{div}(X)$ by $w - u$ and integrate over $\mathbb{X}$ with respect to $\nu$. Using the Green's formula (Theorem \ref{thm:generalgreensformula}), we get
\begin{equation*}
\int_{\mathbb{X}} v(w-u) \, d\nu = - \int_{\mathbb{X}} (w - u) \mbox{div}(X) \, d\nu = \int_{\mathbb{X}} (X,Dw) - \int_{\mathbb{X}} |Du|_\nu.
\end{equation*}
It is clear that $(d)$ implies $(c)$. To finish the proof, let us see that $(c)$ implies $(b)$. If we take $w = u$ in \eqref{eq:variationalinequalitytvflow}, we get
\begin{equation*}
\int_{\mathbb{X}} |Du|_\nu \leq \int_{\mathbb{X}} (X,Du).
\end{equation*}
By Proposition \ref{prop:boundonAnzellottipairing}, this implies that $(X,Du) = |Du|_\nu$ as measures in $\mathbb{X}$.
\end{proof}

\begin{definition}
We define in $L^2({\mathbb{X}},\nu)$ the multivalued operator $\Delta_{1,\nu}$ by
\begin{center}$(u, v ) \in \Delta_{1,\nu}$ \ if and only if, \ $-v \in \partial\mathcal{TV}(u)$.
\end{center}
\end{definition}

Hence, it is natural to introduce the following concept of solutions to the total variation flow in metric measure spaces.

\begin{definition}\label{dfn:totalvariationflow}
{\rm Given $u_0 \in L^2({\mathbb{X}},\nu)$, we say that $u$ is a {\it weak solution} of the Cauchy problem \eqref{Cauchy100} in $[0,T]$, if  { $u \in C([0,T];L^2(\mathbb{X},\nu)) \cap W_{loc}^{1,2}(0, T; L^2(\mathbb{X},\nu))$}, $u(0,\cdot) = u_0$, and for almost all $t \in (0,T)$
\begin{equation}
u_t(t, \cdot) \in \Delta_{1,\nu}(t, \cdot).
\end{equation}
In other words, $u(t) \in BV({\mathbb{X}}, d, \nu)$ and there exist vector fields $X(t) \in  \mathcal{D}^{\infty,2}({\mathbb{X}})$ with $\| X(t) \|_\infty \leq 1$ such that for almost all $t \in [0,T]$ the following conditions hold:
$$ \mbox{div}(X(t)) = u_t(t, \cdot) \quad \hbox{in} \ {\mathbb{X}}; $$
$$ (X(t), Du(t)) = |Du(t)|_\nu \quad \hbox{as measures}.$$
}
\end{definition}

Then, { by the Brezis-Komura Theorem (Theorem \ref{BKTheorem}),} as consequence of Theorem \ref{thm:totalvariationflow}, we have the following existence and uniqueness theorem.

\begin{theorem}\label{EUTV}
For any $u_0 \in L^2(\mathbb{X}, \nu)$ and $T > 0$ there exists a unique weak solution $u(t)$ of the Cauchy problem \eqref{Cauchy100} with $u(0) = u_0$. Moreover, the following comparison principle holds: if $u_1, u_2$ are weak solutions for the initial data $u_{1,0}, u_{2,0} \in  L^2(\mathbb{X}, \nu) \cap  L^r(\mathbb{X}, \nu)$, respectively, then
\begin{equation}\label{CompPrincipleplaplaceBV}
\Vert (u_1(t) - u_2(t))^+ \Vert_r \leq \Vert ( u_{1,0}- u_{2,0})^+ \Vert_r \quad \hbox{for all} \ 1 \leq r \leq \infty.
\end{equation}
We also have
\begin{equation}\label{Regu1}
\left\Vert \frac{d}{dt} u(t) \right\Vert_{L^2(\mathbb{X}, \nu)} \leq \frac{\Vert  u_0 \Vert_{L^2(\mathbb{X}, \nu)}}{t},\quad \hbox{for every } \ t >0,
\end{equation}
and
\begin{equation}\label{Regu2}
\frac{d}{dt} u(t) \leq \frac{u(t)}{t}, \quad \hbox{$\nu$-a.e. on $\mathbb{X}$ for every $t >0$ if $u_0 \geq 0$}.
\end{equation}
\end{theorem}

\begin{proof}
The comparison principle is a consequence of the complete accretivity of the operator $\mathcal{A}_1$. The inequalities \eqref{Regu1} and \eqref{Regu2} are a consequence of \cite[Theorem 4.13]{HM} { (see also \cite{BCr1}).}
\end{proof}

 As before, this definition is consistent with the definition of the gradient flow of the total variation introduced by Ambrosio and di Marino in \cite{ADiM}, the difference being that we provided a precise description of the subdifferential of the total variation functional. Since we have existence and uniqueness of solutions for both definitions, the two notions of solutions to corresponding gradient flows coincide. We again may recover some properties of the gradient flow listed in \cite[Proposition 6.2]{ADiM} directly using Definition \ref{dfn:totalvariationflow} and get some new properties, such as the comparison principle given in Theorem \ref{EUTV}.

As a direct consequence of Corollary \ref{cor:subdifferentialonelaplace}, we also get the following characterisation of weak solutions in terms of variational inequalities.

\begin{corollary}
The following conditions are equivalent: \\
$(a)$ $u$ is a weak solution of the Cauchy problem \eqref{CauchyF1}; \\
$(b)$  { $u \in C([0,T];L^2(\mathbb{X},\nu)) \cap W_{loc}^{1,2}(0, T; L^2(\mathbb{X},\nu))$}, $u(0, \cdot) = u_0$, $u(t) \in BV(\mathbb{X}, d, \nu)$ and there exist vector fields  $X(t) \in  \mathcal{D}^{\infty,2}(\mathbb{X})$ with $\| X(t) \|_\infty \leq 1$ such that for almost all $t \in [0,T]$ we have $\mbox{div}(X(t)) = u_t(t, \cdot)$ in $\mathbb{X}$ and
\begin{equation*}
\int_{\mathbb{X}} u_t (u(t) - w) \, d\nu \leq \int_{\mathbb{X}} (X(t),Dw) - \int_{\mathbb{X}} |Du(t)|_\nu, \quad \forall \, w \in L^2(\mathbb{X},\nu) \cap BV(\mathbb{X},\nu).
\end{equation*}
$(c)$ { $u \in C([0,T];L^2(\mathbb{X},\nu)) \cap W_{loc}^{1,2}(0, T; L^2(\mathbb{X},\nu))$}, $u(0, \cdot) = u_0$, $u(t) \in BV(\mathbb{X}, d, \nu)$ and there exist vector fields  $X(t) \in  \mathcal{D}^{\infty,2}(\mathbb{X})$ with $\| X(t) \|_\infty \leq 1$ such that for almost all $t \in [0,T]$ we have $\mbox{div}(X(t)) = u_t(t, \cdot)$ in $\mathbb{X}$ and
\begin{equation*}
\int_{\mathbb{X}} u_t (u(t) - w) \, d\nu = \int_{\mathbb{X}} (X(t),Dw) - \int_{\mathbb{X}} |Du(t)|_\nu, \quad \forall \, w \in L^2(\mathbb{X},\nu) \cap BV(\mathbb{X},\nu).
\end{equation*}
\end{corollary}

In particular, weak solutions satisfy the evolution variational inequality (see \cite{AGSBook})
\begin{equation*}
\int_{\mathbb{X}} u_t (u-w) \, d\nu \leq \int_{\mathbb{X}} |Dw|_\nu - \int_{\mathbb{X}} |Du|_\nu.
\end{equation*}
Finally, let us note that also the characterisation of the solutions to the total variation flow introduced in this Section agrees with the notion of variational solutions from \cite{LT} and \cite{BDM}. Its variant has been used to study the total variation flow on a bounded domain with Dirichlet boundary conditions in the metric setting in \cite{BCP}. The following Corollary implies that (not counting the complications arising from the boundary condition) weak solutions of the Cauchy problem \eqref{CauchyF1} agree with the definition introduced in \cite{BCP}. In a forthcoming paper, we have established a Green-Gauss formula on open bounded sets in metric measure spaces and we intend to apply it on bounded domains with either Dirichlet or Neumann boundary data; moreover, we intend to allow the initial data to lie in $L^1$ using a notion of entropy solutions.

\begin{corollary}
Suppose that $u$ is a weak solution of the Cauchy problem \eqref{CauchyF1}. Then, for any $v \in L^1_w(0,T; BV(\mathbb{X}))$ with $\partial_t v \in L^2(\mathbb{X} \times [0,T])$ and $v(0) \in L^2(\mathbb{X},\nu)$ we have
\begin{equation*}
\int_0^T \int_{\mathbb{X}} \partial_t v(v-u) \, d\nu \, dt + \int_0^T \int_{\mathbb{X}} |Dv(t)|_\nu - \int_0^T \int_{\mathbb{X}} |Du(t)|_\nu
\end{equation*}
\begin{equation}\label{varineq}
\geq \frac12 \Vert (v-u)(T) \Vert^2_{L^2(\mathbb{X},\nu)} - \frac12 \Vert v(0)-u_0 \Vert^2_{L^2(\mathbb{X},\nu)}.
\end{equation}
 Note that by \cite[Theorem 3.2]{Brezis} the weak solutions also have this regularity.
\end{corollary}

\begin{proof}
Given a test function $v$ as above, we want to show that \eqref{varineq} holds. We start by computing the term with the time derivative using the characterisation of weak solutions. Applying Green formula we have
$$\int_0^T \int_{\mathbb{X}} \partial_t u (v - u) d \nu dt = \int_0^T \int_{\mathbb{X}} \mbox{div}(X(t)) (v - u) d \nu dt $$
$$  = - \int_0^T \int_{\mathbb{X}} (X(t), Dv(t)) dt + \int_0^T \int_{\mathbb{X}} (X(t), Du(t)) dt.$$
Also,
$$\int_0^T \int_{\mathbb{X}} (\partial_t v - \partial_t u )(v - u) d \nu dt = \frac12 \Vert (v-u)(T) \Vert^2_{L^2(\mathbb{X},\nu)} - \frac12 \Vert v(0)-u_0 \Vert^2_{L^2(\mathbb{X},\nu)}. $$
Since $u$ is a weak solution, adding the two equalities we get
$$\int_0^T \int_{\mathbb{X}} \partial_t v (v - u) d \nu dt $$
$$= - \int_0^T \int_{\mathbb{X}} (X(t), Dv(t)) dt + \int_0^T \int_{\mathbb{X}} (X(t), Du(t)) dt + \frac12 \Vert (v-u)(T) \Vert^2_{L^2(\mathbb{X},\nu)} - \frac12 \Vert v(0)-u_0 \Vert^2_{L^2(\mathbb{X},\nu)} $$
$$ \geq - \int_0^T \vert Dv(t)\vert_\nu (\mathbb{X}) dt + \int_0^T \vert Du(t)\vert_\nu (\mathbb{X}) dt + \frac12 \Vert (v-u)(T) \Vert^2_{L^2(\mathbb{X},\nu)} - \frac12 \Vert v(0)-u_0 \Vert^2_{L^2(\mathbb{X},\nu)}.$$
\end{proof}

\section{Asymptotic Behaviour}\label{sec:asymptotics}

Let $\mathcal{H}$ be a Hilbert space and $J: \mathcal{H} \rightarrow ]-\infty, + \infty]$ a proper, convex, lower semi-continuous functional. Then, it is well known (see \cite{Brezis}) that the abstract Cauchy problem
\begin{equation}\label{ACP123}
\left\{ \begin{array}{ll} u'(t) + \partial J(u(t)) \ni 0, \quad t \in [0,T] \\[5pt] u(0) = u_0, \end{array}\right.
\end{equation}
has a unique strong solution $u(t)$ for any initial datum $u_0 \in \overline{D(J)}$.

For $p \geq 1$, we say that $J$ is {\it $p$-homogenous} if
$$J(\lambda u) = \vert \lambda \vert^p J(u), \quad \forall \, \lambda \not=0, \ u \in \mathcal{H} \quad \hbox{and} \quad  J(0) =0,$$
and we say that $J$ is {\it $p$-coercive} if there exists a constant $C >0$ such that
\begin{equation}\label{coerc}
\Vert u \Vert^p \leq C J(u), \quad \forall \, u\in \mathcal{H}_0,
\end{equation}
where
$$\mathcal{H}_0:= \{ u \in \mathcal{H} \ : \ J(u)=0 \}^{\perp} \setminus \{ 0 \}.$$
Obviously, this inequality is equivalent to positive lower bound of the Rayleigh quotient associated with $J$, i.e.,
$$\lambda_1(J):= \inf_{u \in \mathcal{H}_0} \frac{pJ(u)}{\Vert u \Vert^p} > 0.$$

For $u_0 \in \mathcal{H}_0$, if $u(t)$ is the strong solution of \eqref{ACP123}, we define its {\it extinction time} as
$$T_{\rm ex}(u_0):= \inf \{ T >0 \ : \ u(t) =0,  \ \ \forall \, t \geq T \}.$$

In the next result, we summarize the results obtained by Bungert and Burger in \cite{BB}.

\begin{theorem}\label{mainResult} Let $J$ be a convex, lower-semicontinuous functional on  $\mathcal{H}$ with dense domain. Assume that $J$ is $p$-homogeneous and coercive. For  $u_0 \in \mathcal{H}_0$, let  $u(t)$ be the strong solution of \eqref{ACP123}. Then, we have
\begin{itemize}
\item[(i)] (Finite extinction time) For $1 \leq p <2$,
$$T_{\rm ex}(u_0) \leq \frac{\Vert u_0 \Vert^{p-2}}{(2-p)\lambda_1(J)}.$$

\item[(ii)] (Infinite extinction time) For $ p \geq 2$,
$$T_{\rm ex}(u_0) = +\infty.$$

\item[(iii)] (General upper bounds)
$$\Vert u(t) \Vert^{2-p} \leq \Vert u_0 \Vert^{2-p} - (2-p)\lambda_1(J) t, \quad  1 \leq p < 2,$$
$$\Vert u(t) \Vert^{2} \leq \Vert u_0 \Vert^{2} e^{-2\lambda_1(J) t}, \quad   p = 2,$$
$$\Vert u(t) \Vert^{2-p} \leq \frac{1}{\Vert u_0 \Vert^{2-p} + (2-p)\lambda_1(J) t}, \quad p>2.$$

\item[(iv)] (Sharper bound for the finite extinction)  For $1 \leq p <2$,
$$(2-p)\lambda_1(J) (T_{\rm ex}(u_0) - t) \leq \Vert u(t) \Vert \leq (2-p)\Lambda (t)(T_{\rm ex}(u_0) - t),  $$
where
  $$\Lambda(t):= \frac{J(u(t))}{\Vert u(t) \Vert}.$$
\item[(v)] (Asymptotic profiles for finite extinction) For $1 \leq p <2$, let $$w(t):= \frac{u(t)}{\left(1 - \frac{1}{T_{\rm ex}(u_0)} t \right)^{\frac{1}{2-p}}},$$
    and assume that $w(t)$ converges (possibly up to a subsequence) strongly to some $w_* \in \mathcal{H}$ as $t \to T_{\rm ex}(u_0)$. Then it holds $$ \frac{1}{T_{\rm ex}(u_0)} \in \partial J(w_*),\quad   w_*  \not=0, \quad  \Vert w_* \Vert \leq \Vert u_0 \Vert.$$
    \item[(vi)] (Ground state as asymptotic profile) For $1 \leq p <2$, an asymptotic profile $w_*$ is a ground state, i.e., $ w_* = {\rm argmin} \frac{J(w_*)}{\Vert w_* \Vert}$, if and only if $\lim_{t \nearrow T_{\rm ex}}(u_0) = \lambda_1(J).$
\end{itemize}

\end{theorem}

Now we are going to apply Theorem \ref{mainResult} to study the asymptotic behaviour of the weak solutions of the Cauchy problems \eqref{eq:p-Lapcauchy} and \eqref{CauchyF1}.

Obviously, the convex, lower semi-continuous functionals $\mathsf{Ch}_p$ are $p$-homogeneous.

In the case $\nu(\mathbb{X}) < \infty$, we have that $\mathsf{Ch}_p$ is coercive if and only if if there exists a constant $M >0$ such that
\begin{equation}\label{coercDfini}
\Vert u \Vert_{L^2(\mathbb{X}, \nu)}^p \leq M \, \mathsf{Ch}_p(u), \quad \forall \, u\in \left\{ u \in L^2(\mathbb{X}, \nu), \ u \not=0,  \ : \int_\mathbb{X} u \, d\nu =0\ \right\},
\end{equation}

which is equivalent to the following {\it Poincar\'{e} inequality}

\begin{equation}\label{Poincare1}
\Vert u - \overline{u} \Vert_{L^2(\mathbb{X}, \nu)}^p \leq M \, \mathsf{Ch}_p(u) \quad \forall \, u \in W^{1,p}(\mathbb{X}, d, \nu) \cap L^2(\mathbb{X}, \nu),
\end{equation}
 where
 $$\overline{u}:= \frac{1}{\nu(\mathbb{X})} \int_\mathbb{X} u d\nu$$
 for $1 <p <\infty$; and

\begin{equation}\label{Poincare1p=1}
\Vert u - \overline{u} \Vert_{L^2(\mathbb{X}, \nu)} \leq M \, \mathcal{TV}(u) \quad \forall \, u \in BV(\mathbb{X}, d, \nu) \cap L^2(\mathbb{X}, \nu),
\end{equation}
for $p = 1.$

In the case $\nu(\mathbb{X}) =+\infty$, we have that $\mathsf{Ch}_p$ is coercive if and only if the following {\it Sobolev inequality} holds:  there  exists a constant $M >0$ such that
\begin{equation}\label{Sobolev1p}
\Vert u \Vert_{L^2(\mathbb{X}, \nu)}^p \leq M \, \mathsf{Ch}_p(u) \quad \forall \, u \in L^2(\mathbb{X}, \nu). \end{equation}

{ Then, if we assume that \eqref{Poincare1} or \eqref{Poincare1p=1} holds, we have
$$L^2(\mathbb{X}, \nu)_0:= \{ u \in L^2(\mathbb{X}, \nu) \ : \ \mathsf{Ch}_p(u)=0 \}^{\perp} \setminus \{ 0 \} = \left\{ u \in L^2(\mathbb{X}, \nu), \ u \not=0,  \ : \int_\mathbb{X} u \, d\nu =0\ \right\},$$
in the case $\nu(\mathbb{X}) < \infty$, and if we assume that \eqref{Sobolev1p} holds, we have
$$L^2(\mathbb{X}, \nu)_0 = L^2(\mathbb{X}, \nu) \setminus \{ 0 \},$$
in the case $\nu(\mathbb{X}) = +\infty$.}

As a consequence of Theorem \ref{mainResult}, we have the following results.

\begin{theorem}\label{mainResultfinit}  Assume that $\nu(\mathbb{X}) < \infty$ and the  Poincar\'{e} inequality \eqref{Poincare1} holds, for $1 <p <\infty$ and \eqref{Poincare1p=1}, for $p=1$. For $u_0 \in L^2(\mathbb{X}, \nu)$, let $u(t)$ be the weak solution of the Cauchy problem \eqref{eq:p-Lapcauchy}, for $1 < p < \infty$, and the weak solution of the Cauchy problem \eqref{CauchyF1}, for $p=1$. Then, we have
\begin{itemize}
\item[(i)] (Finite extinction time) For $1 \leq p <2$,
$$T_{\rm ex}(u_0) \leq \frac{\Vert u_0 \Vert_{ L^2(\mathbb{X}, \nu)}^{p-2}}{(2-p)\lambda_1(\mathsf{Ch}_p)},$$
where
$$T_{\rm ex}(u_0):= \inf \{ T >0 \ : \ u(t) = \overline{u_0},  \ \ \forall \, t \geq T \}.$$

\item[(ii)] (Infinite extinction time) For $p \geq 2$,
$$T_{\rm ex}(u_0) = +\infty.$$

\item[(iii)] (General upper bounds)
$$\Vert u(t) -\overline{u_0} \Vert_{ L^2(\mathbb{X}, \nu)}^{2-p} \leq \Vert u_0 \Vert_{ L^2(\mathbb{X}, \nu)}^{2-p} - (2-p)\lambda_1(\mathsf{Ch}_p) t, \quad  1 \leq p < 2,$$
$$\Vert u(t)-\overline{u_0} \Vert_{ L^2(\mathbb{X}, \nu)}^{2} \leq \Vert u_0 \Vert_{ L^2(\mathbb{X}, \nu)}^{2} e^{-2\lambda_1(\mathsf{Ch}_p) t}, \quad   p = 2,$$
$$\Vert u(t) -\overline{u_0}\Vert_{ L^2(\mathbb{X}, \nu)}^{2-p} \leq \frac{1}{\Vert u_0 \Vert_{ L^2(\mathbb{X}, \nu)}^{2-p} + (2-p)\lambda_1(\mathsf{Ch}_p) t}, \quad p>2.$$

\item[(iv)] (Sharper bound for the finite extinction)  For $1 \leq p <2$,
$$(2-p)\lambda_1(\mathsf{Ch}_p) (T_{\rm ex}(u_0) - t) \leq \Vert u(t) -\overline{u_0} \Vert_{ L^2(\mathbb{X}, \nu)} \leq (2-p)\Lambda (t)(T_{\rm ex}(u_0) - t),  $$
where
  $$\Lambda(t):= \frac{\mathsf{Ch}_p(u(t))}{\Vert u(t) \Vert}.$$
\item[(v)] (Asymptotic profiles for finite extinction) For $1 \leq p <2$, let $$w(t):= \frac{u(t)}{\left(1 - \frac{1}{T_{\rm ex}(u_0)} t \right)^{\frac{1}{2-p}}},$$
    and assume that $w(t)$ converges (possibly up to a subsequence) strongly to some $w_* \in  L^2(\mathbb{X}, \nu)$ as $t \to T_{\rm ex}(u_0)$. Then it holds $$ \frac{1}{T_{\rm ex}(u_0)} \in \partial \mathsf{Ch}_p(w_*),\quad   w_*  \not=0, \quad  \Vert w_* \Vert_{ L^2(\mathbb{X}, \nu)} \leq \Vert u_0 \Vert_{ L^2(\mathbb{X}, \nu)}.$$
    \item[(vi)] (Ground state as asymptotic profile) For $1 \leq p <2$, an asymptotic profile $w_*$ is a ground state, i.e., $ w_* = {\rm argmin} \frac{\mathsf{Ch}_p(w_*)}{\Vert w_* \Vert_{ L^2(\mathbb{X}, \nu)}}$, if and only if $\lim_{t \nearrow T_{\rm ex}}(u_0) = \lambda_1(\mathsf{Ch}_p).$
\end{itemize}

\end{theorem}
\begin{proof} It is a direct application of Theorem \ref{mainResult}, having in mind that for any constant function $v_0$ and any $u_0 \in L^2(\mathbb{X}, \nu)$, we have   $\mathsf{Ch}_p(u_0 + v_0) = \mathsf{Ch}_p(u_0)$ and  $\partial \mathsf{Ch}_p(u_0 + v_0) = \partial \mathsf{Ch}_p(u_0)$ (see \cite[Proposition A.3]{BB}).

\end{proof}

The next result is a direct consequence  of Theorem \ref{mainResult}.

\begin{theorem}\label{mainResultinfinit}  Assume that $\nu(\mathbb{X}) =+\infty$ and the Sobolev inequality \eqref{Sobolev1p} holds. For $u_0 \in L^2(\mathbb{X}, \nu)$, let  $u(t)$ be the weak  solution of the Cauchy problem \eqref{eq:p-Lapcauchy}, for $1 < p < \infty$, and the  weak solution of the Cauchy problem \eqref{CauchyF1}, for $p=1$. Then, we have
\begin{itemize}
\item[(i)] (Finite extinction time) For $1 \leq p <2$,
$$T_{\rm ex}(u_0) \leq \frac{\Vert u_0 \Vert_{ L^2(\mathbb{X}, \nu)}^{p-2}}{(2-p)\lambda_1(\mathsf{Ch}_p)},$$
where
$$T_{\rm ex}(u_0):= \inf \{ T >0 \ : \ u(t) = 0,  \ \ \forall \, t \geq T \}.$$

\item[(ii)] (Infinite extinction time) For $p \geq 2$,
$$T_{\rm ex}(u_0) = +\infty.$$

\item[(iii)] (General upper bounds)
$$\Vert u(t) \Vert_{ L^2(\mathbb{X}, \nu)}^{2-p} \leq \Vert u_0 \Vert_{ L^2(\mathbb{X}, \nu)}^{2-p} - (2-p)\lambda_1(\mathsf{Ch}_p) t, \quad  1 \leq p < 2,$$
$$\Vert u(t) \Vert_{ L^2(\mathbb{X}, \nu)}^{2} \leq \Vert u_0 \Vert_{ L^2(\mathbb{X}, \nu)}^{2} e^{-2\lambda_1(v) t}, \quad   p = 2,$$
$$\Vert u(t) \Vert_{ L^2(\mathbb{X}, \nu)}^{2-p} \leq \frac{1}{\Vert u_0 \Vert_{ L^2(\mathbb{X}, \nu)}^{2-p} + (2-p)\lambda_1(\mathsf{Ch}_p) t}, \quad p>2.$$

\item[(iv)] (Sharper bound for the finite extinction)  For $1 \leq p <2$,
$$(2-p)\lambda_1(\mathsf{Ch}_p) (T_{\rm ex}(u_0) - t) \leq \Vert u(t) \Vert_{ L^2(\mathbb{X}, \nu)} \leq (2-p)\Lambda (t)(T_{\rm ex}(u_0) - t),  $$
where
  $$\Lambda(t):= \frac{\mathsf{Ch}_p(u(t))}{\Vert u(t) \Vert_{L^2(\mathbb{X}, \nu)} }.$$
\item[(v)] (Asymptotic profiles for finite extinction) For $1 \leq p <2$, let $$w(t):= \frac{u(t)}{\left(1 - \frac{1}{T_{\rm ex}(u_0)} t \right)^{\frac{1}{2-p}}},$$
    and assume that $w(t)$ converges (possibly up to a subsequence) strongly to some $w_* \in  L^2(\mathbb{X}, \nu)$ as $t \to T_{\rm ex}(u_0)$. Then it holds $$ \frac{1}{T_{\rm ex}(u_0)} \in \partial \mathsf{Ch}_p(w_*),\quad   w_*  \not=0, \quad  \Vert w_* \Vert_{L^2(\mathbb{X}, \nu)}  \leq \Vert u_0 \Vert_{L^2(\mathbb{X}, \nu)} .$$
    \item[(vi)] (Ground state as asymptotic profile) For $1 \leq p <2$, an asymptotic profile $w_*$ is a ground state, i.e., $ w_* = {\rm argmin} \frac{\mathsf{Ch}_p(w_*)}{\Vert w_* \Vert_{L^2(\mathbb{X}, \nu)} }$, if and only if $\lim_{t \nearrow T_{\rm ex}}(u_0) = \lambda_1(\mathsf{Ch}_p).$
\end{itemize}

\end{theorem}

\begin{remark}{\rm
In the monographs  \cite{BjBj}, \cite{Haj}, \cite{HKST}, \cite{Herbey} one can find many important examples of metric measures spaces, including the weighted Euclidean spaces, Riemannian manifolds, Carnot-Carath\'{e}odory spaces, Alexandrov spaces etc., which satisfy the Poincar\'{e} and Sobolev inequalities that we need to apply Theorems \ref{mainResultfinit} and \ref{mainResultinfinit}.}$\blacksquare$
\end{remark}

\section{Some important particular cases}\label{sec:particularcases}

In this section we are going to apply our general results to some important metric measure spaces and to see what is the definition of the $p$-Laplacian operator in these particular metric measure spaces.

\subsection{$p$-Laplacian in Weighted Euclidean Spaces} Endow $\R^N$ with the Euclidean distance $d_{Eucl}$. For a nonnegative Radon measure $\nu$  in $(\R^N, d_{Eucl})$, we refer to the metric measure space $(\R^N,  d_{Eucl},\nu)$ as a {\it weighted Euclidean space}.  For $1 < p < \infty$, $\frac{1}{p} + \frac{1}{q}=1$,  we shall denote the cotangent module by $L^p_\nu(T^*\R^N)$, and the tangent module by $L^q_\nu(T \R^N)$.

Given a metric measure space $(\mathbb{X}, d, \nu)$ and  a Banach space $(\mathbb{B}, \Vert \cdot \Vert)$ we denote by $L^p(\mathbb{X},\mathbb{B}, \nu)$ the set of all Borel maps $v: \mathbb{X} \rightarrow \mathbb{B}$ such that
$$\int_\mathbb{X} \Vert v(x) \Vert^p \, d\nu(x) < \infty.$$
$L^p(\mathbb{X},\mathbb{B}, \nu)$ is a Banach space respect to the norm
$$\left( \int_\mathbb{X} \Vert v(x) \Vert^p \, d\nu(x) \right)^{\frac{1}{p}}.$$
Moreover, $L^p(\mathbb{X},\mathbb{B}, \nu)$ is a $L^p(\nu)$-normed module when endowed with the natural pointwise operations and the the following pointwise norm: given any $v \in L^p(\mathbb{X},\mathbb{B}, \nu)$, we define
$$\vert v(x) \vert:= \Vert v(x) \Vert, \quad \hbox{for $\nu$-a.e.} \ x \in \mathbb{X}.$$

Following \cite{LPR}, we called to the elements of $L^q(\R^N, \R^N, \nu)$ as the {\it concrete vector fields} on $(\R^N,  d_{Eucl},\nu)$. The module $L^p(\R^N, (\R^N)^*, \nu)$ is the dual module of $L^q(\R^N, \R^N, \nu)$ and its elements are said to be the {\it concrete $1$-forms} on $(\R^N,  d_{Eucl},\nu)$.

To distinguish between the classically defined differential and the
one coming from the theory of modules, we shall denote the former by $\underline{d}f$, while keeping $df$ for the latter. More generally, elements of $L^p(\R^N, (\R^N)^*, \nu)$ or $L^q(\R^N, \R^N, \nu)$ will typically be
underlined, while those of $L^p_\nu(T^*\R^N)$, $L^q_\nu(T \R^N)$ will be not. The strong differential of a given function $f \in C^\infty_c(\R^N)$ will be denote by $\underline{d}f \in L^q(\R^N, (\R^N)^*, \nu)$.

We denote by $D^{q,r}(\underline{{\rm div}}_\nu)$ the space of all vector fields $\underline{X} \in L^q(\R^N, \R^N, \nu)$ whose distributional divergence belongs to $L^r(\R^N, \nu)$. Namely, there exists a function $\underline{{\rm div}}_\nu(\underline{X}) \in L^r(\R^N, \nu)$ such that
$$\int_{\R^N} \nabla \varphi \cdot \underline{X} \, d \nu = - \int_{\R^N} \varphi \, \underline{{\rm div}}_\nu(\underline{X}) \, d \nu, \quad \hbox{for every} \ \  \varphi \in C_c^\infty(\R^N).$$

The following result was proved for $p=q=2$ in \cite{GP} (see also \cite{LPR}),  but the proof also works for the case $1 < p < \infty$, so we have the following result.

\begin{theorem} There exists a unique surjective morphism $P_\nu: L^p(\R^N, (\R^N)^*, \nu) \rightarrow  L^p_\nu(T^*\R^N)$ such that
$$P_\nu( \underline{d} f) = df, \quad \hbox{for every} \ f \in C^\infty_c(\R^N),$$
and
\begin{equation}\label{good1}\vert \omega \vert_{*} = \min_{\underline{\omega} \in P_\nu^{-1}(\omega)} \vert \underline{\omega} \vert, \quad \hbox{ $\nu$-a.e.} \quad \forall \, \omega \in L_\nu^p(T^*\R^N).
          \end{equation}
Denote by $i_\nu : L^q_\nu(T \R^N) \rightarrow L^q(\R^N, \R^N, \nu)$ the adjoint of $P_\nu$, i.e., the unique morphism satisfying
$$P_\nu (\underline{\omega}) (v) = \underline{\omega}(i_\nu(v)), \quad \hbox{for every} \ v \in  L^q_\nu(T \R^N)  \ \hbox{and} \ \ \underline{\omega} \in L^q(\R^N, \R^N, \nu).$$
Then, we have that
$$\vert i_\nu (v) \vert = \vert v \vert \quad \hbox{holds $\nu$-a.e. \ on} \ \R^N,\quad \hbox{for any given} \ v \in  L^q_\nu(T \R^N).$$

\end{theorem}

Now we are going to characterize the $p$-Laplacian operator in the  weighted Euclidean space $(\R^N,  d_{Eucl},\nu)$.

\begin{theorem}\label{CharactWES} Let $1 < p < \infty$ be and let $\Delta_{p,\nu}$ be the $p$-Laplacian  in  $(\R^N,  d_{Eucl},\nu)$. For $v \in L^2(\R^N, \nu)$, we have $-\Delta_{p,\nu} u = v$ if and only if $u \in W^{1,p}(\R^N,d_{Eucl},\nu)$ and there exists a concrete vector field  $\underline{X} \in D^{q,2}(\underline{{\rm div}}_\nu)$ satisfying the following condition:
\begin{equation}\label{E1Chat}
\Vert \underline{X}  \Vert^q  \leq \vert Du \vert^p_\nu \quad \hbox{$\nu$-a.e. in} \ \R^N,
\end{equation}
\begin{equation}\label{E2Chat}
- \underline{{\rm div}}_\nu(\underline{X}) = v\quad \hbox{in} \ \R^N,
\end{equation}
\begin{equation}\label{E3Chat}
\hbox{there exists} \ \underline{u} \in P_\nu^{-1}(du) \in L^p(\R^N, (\R^N)^*, \nu), \ \hbox{such that} \ \underline{u}(\underline{X}) = \vert \underline{u} \vert^p \ \hbox{$\nu$-a.e. \ in} \ \R^N.
\end{equation}

\end{theorem}
\begin{proof} Suppose that $-\Delta_{p,\nu} u = v$. Then, $u \in W^{1,p}(\R^N,  d_{Eucl},\nu)$ and there exists a vector field $X \in \mathcal{D}^{q,2}(\R^N)$ with  $| X |^q \leq | du |_*^p$ $\nu$-a.e. such that
\begin{equation}\label{E01Chat} -\mbox{div}(X) = v, \quad  \hbox{in $\R^N$},
\end{equation}
\begin{equation}\label{E02Chat}
du(X) = |du|_*^p  \quad \nu\hbox{-a.e. in } \mathbb{\R^N}.
\end{equation}
Since $X \in L^q_\nu(T \R^N)$, we have  $\underline{X}:= i_\nu(X) \in L^q(\R^N, \R^N, \nu)$. Then,
$$\Vert \underline{X}  \Vert^q = \vert  i_\nu(X) \vert^q = \vert X \vert^q \leq | du |_*^p =\vert Du \vert^p_\nu \quad \hbox{$\nu$-a.e. in} \ \R^N, $$
and \eqref{E1Chat} holds.

For $\varphi \in C_c^\infty(\R^N)$, by \eqref{E01Chat}, we have
$$\int_{\R^N} v \varphi \, d \nu  = - \int_{\R^N} \varphi \,\mbox{div}(X) \, d\nu =  \int_{\R^N} d\varphi(X) \, d\nu = \int_{\R^N} \nabla \varphi \cdot \underline{X} \, d \nu = - \int_{\R^N} \varphi \, \underline{{\rm div}}_\nu(\underline{X}) \, d \nu, $$
and \eqref{E2Chat} holds.

By \eqref{good1}, there exists $\underline{u} \in P_\nu^{-1}(du)$ such that $\vert \underline{u} \vert = \vert du \vert_*$. Then, by \eqref{E02Chat}, we have for $\nu$-a.e. in $\R^N$,
$$\vert \underline{u} \vert^p =  \vert du \vert_*^p = du(X) = P_\nu(\underline{u})(X) = \underline{u}(i_\nu (X)) = \underline{u}(\underline{X}),$$
and \eqref{E3Chat} holds.

Reciprocally, suppose that $u \in W^{1,p}(\R^N,d_{Eucl},\nu)$ and there exists a concrete vector field $\underline{X} \in  D^{q,2}(\underline{{\rm div}}_\nu)$ satisfying \eqref{E1Chat}, \eqref{E2Chat} and \eqref{E3Chat}. Let $X \in i_\nu^{-1}(\underline{X}) \in L^q_\nu(T\R^N)$, then
$$\vert X \vert^q = \vert  i_\nu(X) \vert^q = \Vert \underline{X}  \Vert^q \leq \vert Du \vert^p_\nu = | du |_*^p.$$
Given $\varphi \in C_c^\infty(\R^N)$, by \eqref{E2Chat}, we have
$$\int_{\R^N} v \varphi \, d\nu = \int_{\R^N} \nabla \varphi \cdot \underline{X} \, d\nu = \int_{\R^N} d\varphi(X) \, d\nu =  - \int_{\R^N}\varphi \, \mbox{div}(X) \, d\nu.$$
Now, since $C_c^\infty(\R^N)$ is dense in $W^{1,p}(\R^N,  d_{Eucl},\nu)$, we have that given $g \in W^{1,p}(\R^N,  d_{Eucl},\nu)$,
$$\int_{\R^N} v g \, d\nu = - \int_{\R^N} g \, {\rm div} (X) \, d\nu = \int_{\R^N} dg(X) \, d\nu,$$
and therefore \eqref{E01Chat} holds.

Finally, by \eqref{E3Chat}, we have for $\nu$-a.e. in $\R^N$,
$$\vert du \vert_*^p = \vert \underline{u} \vert^p = P_\nu(\underline{u})(X) = du(X),$$
and \eqref{E02Chat} holds.
\end{proof}

{ Using this characterisation of the operator $\Delta_{p,\nu}$, one can immediately obtain a corresponding characterisation of solutions to the Cauchy problem for the $p$-Laplacian in the weighted Euclidean space $(\R^N,  d_{Eucl},\nu)$ in terms of concrete vector fields.}

\begin{remark}{\rm The above result is new even for the particular case when the weight is absolutely continuous with respect to the Lebesgue measure, i.e. $\nu = \omega \mathcal{L}^N$. To the best of our knowledge, in this particular case we only know the results by T\"olle \cite{Tolle} for weight $\omega = \varphi^p$, with $p \frac{\nabla \varphi}{\varphi} \in L^q_{loc}(\R^N, \nu)$. $\blacksquare$
}
\end{remark}

\subsection{$p$-Laplacian in Finsler Manifolds}

Here we shall assume that $M$ is a $N$-dimensional, connected differentiable manifold of class $C^\infty$. Given $x \in M$, we denote by $T_xM$ the tangent space of $M$ at $x$. We denote $TM:= \sqcup_{x \in M} T_xM$ is the tangent bundle of $M$. Moreover, we denote by $T^*_xM$ and $T^*M$ the cotangent space of $M$ at $x$ and the cotangent bundle of $M$, respectively.

Let $V$ be a finite-dimensional vector space over $\R$. A {\it Minkowski norm} on $V$ is any functional $F: V \rightarrow [0, +\infty)$ satisfying:
\begin{itemize}
\item[(1)] $F(v)= 0 \iff v=0,$
\item[(2)] $F(u + v) \leq F(u) + F(v), \quad \forall \, u,v \in V,$
\item[(3)] $F(\lambda v) = \lambda F(v), \quad \forall \,v \in V \ \hbox{and} \ \lambda \geq 0$,
\item[(4)] $F$ is continuous on $V$ and of class $C^\infty$ on $V \setminus \{ 0 \}$,
\item[(5)] Given any $v \in V \setminus \{ 0 \}$, it holds that the quadratic form
$$V \in w \mapsto \frac12 d^2(F^2)v [w,w]$$
is positive definite, where $d^2$ is the second differential.
\end{itemize}

\begin{definition}{\rm  A {\it Finsler manifold} is any couple $(M,F)$, where $M$ is a given manifold and $F: TM \rightarrow [0, +\infty)$ is a continuous function (called a {\it Finsler structure}) satisfying the following properties:
\begin{itemize}
\item[(i)] The function $F$ is of class $C^\infty$ on $TM \setminus \{ 0 \}$.

\item[(ii)] The functional $F(x, \cdot): T_xM \rightarrow [0, +\infty)$ is a Minkowski norm for every $x \in M$.
\end{itemize}

We say that $(M,F)$ is {\it reversible} if $F(x, \cdot)$ is a symmetric norm on $T_x M$ for any $x \in M$. Moreover, for a Finsler structure $F$ on $M$, we define the dual structure $F^* : T^*M \rightarrow [0, \infty)$ by
$$F^*(x, \alpha):= \sup \{ \alpha \xi \ : \xi \in T_x M,  \ F(x, \xi) \leq 1 \}. $$
We remark that $F^*(x, \cdot)$ is a Minkowski norm on $T^*_x M$.
}
\end{definition}

If $(M,g)$ is a Riemannian manifold, then it is a reversible Finsler manifold $(M,F)$, where $F(x, v):= \sqrt{g_x(v,v)}$.

\begin{definition}{\rm Let $(M,F)$ be a reversible Finsler manifold. Given a piecewise $C^1$ curve $\gamma: [0,1] \rightarrow M$, we define its {\it Finsler length} as
$$\ell_F(\gamma):= \int_0^1 F(\gamma(t),\dot{\gamma}(t)) \, dt.$$
Then, we define the {\it Finsler distance} between two points $x,y \in M$ as
$$d_F(x,y):= \inf \left\{\ell_F(\gamma) \ : \ \gamma: [0,1] \rightarrow M \ \hbox{piecewise} \ C^1 \ \hbox{with}  \ \gamma(0) = x, \ \gamma(1)=y \right\}.$$
}
\end{definition}

From now on we will assume that $(M,F)$ is a geodesically complete, reversible Finsler manifold. Then, by the Hopf-Rinow Theorem (\cite[Theorem 6.6.1]{BCS}), we have that the metric space $(M, d_F)$ is complete and proper. Therefore, if $\nu$ is non-negative Radon measure on $(M, d_F)$, the metric measure space $(M, d_F, \nu)$ satisfies the assumptions used in Section \ref{sec:plaplace}. For $1 < p < \infty$, $\frac{1}{p} + \frac{1}{q}=1$,  we shall denote the cotangent module associated to $(M, d_F, \nu)$ by $L^p_\nu(T^*M)$, and the tangent module associated to $(M, d_F, \nu)$ by $L^q_\nu(T M)$. Following Lu$\check{c}$i\'c and Pasqualetto \cite{LP}, we are going to see that $L^q_\nu(T M)$ can be isometrically embedding into the spaces of all measurable sections of the tangent bundle of $M$ that are $q$-integrable with respect to $\nu$.

Given $f \in C^1(M)$, we denote by $\underline{d}f$ its differential, which is a continuous section of the cotangent bundle $T^*M$. We set
$$\vert \underline{d}f \vert(x):= F^*(x, \underline{d}f(x)) \quad \hbox{for every} \ x \in M,$$
where $F^*(x, \cdot)$ stand for the dual norm of $F(x, \cdot)$. We also can consider the function $f$ as an element to the Sobolev space $W^{1,p}(M, d_F, \nu)$ and we have
\begin{equation}\label{iinneq}
 \vert df \vert = \vert Du \vert_\nu \leq \vert \nabla u \vert = \vert \underline{d}f \vert \quad \hbox{$\nu$-a.e. on } M.
\end{equation}

We define the concrete tangent/cotangent module associated to $(M, d_F, \nu)$ as
$$\Gamma_q(TM;\nu):= \hbox{space of all  \ $L^q(\nu)$-sections of} \ TM,$$
i.e.
$$\Gamma_q(TM;\nu):= \left\{ v : M \rightarrow TM \ : \ v(x) \in T_xM, \ \int_M F(x,v(x))^q \, d\nu(x) < \infty \right\}.$$
$$\Gamma_p(T^*M;\nu):= \hbox{space of all  \ $L^p(\nu)$-sections of} \ T^*M,$$
i.e.
$$\Gamma_p(T^*M;\nu):= \left\{ \underline{\omega} : M \rightarrow T^*M \ : \ \underline{\omega}(x) \in T^*_xM, \ \int_M F^*(x,\underline{\omega}(x))^p \, d\nu(x) < \infty \right\}.$$

The space $ \Gamma_q(TM;\nu)$ has a natural structure of an $L^q(\nu)$-normed module if we endow it with the usual vector structure and the following pointwise operations:
$$(fv)(x):= f(x) v(x) \in T_xM, \quad \vert v \vert(x):= F(x,v(x)), \quad \hbox{for $\nu$-a.e. } \ x \in M$$
for any $v \in \Gamma_q(TM;\nu)$ and $f \in L^\infty(M, \nu)$. Similarly, $\Gamma_p(T^*M;\nu)$ is an $ L^p(\nu)$-normed module. Moreover, $\Gamma_q(TM;\nu)$ and $\Gamma_p(T^*M;\nu)$ are the module duals of the other (for $1 < p < \infty$). Moreover, we have that (see \cite{LP})
$$\{ \underline{d} f \ : \ f \in C^1_c(M) \} \ \hbox{generates}  \ \Gamma_p(T^*M;\nu) \ \hbox{in the sense of modules},$$
where each element  $\underline{d} f$ can be viewed as an element of $\Gamma_p(T^*M;\nu) $ as it is a continuous section of the cotangent module $T^*M$ and its associated pointwise norm $\vert \underline{d} f \vert$ has compact support.

The following result was proved in \cite{LP} for the case $p= q =2$, but the proof is also valid for $1 < p < \infty$.

Given a vector field $\underline{X} \in \Gamma_q(TM;\nu)$, we define its {\it divergence} $\underline{{\rm div}}_\nu : M \rightarrow \R$ through the identity
$$\int_M f \, \underline{{\rm div}}_\nu(\underline{X}) \, d \nu = - \int_M \underline{d}f (\underline{X}) \, d \nu, $$
for all $f \in C^\infty_c(M)$, where $\underline{d}f (\underline{X})$ at $x \in M$ denotes the canonical paring between $T^*_x M$ and $T_xM$. We denote by $D^{q,2}(\underline{{\rm div}}_\nu)$ the space of all vector fields $\underline{X} \in \Gamma_q(TM;\nu)$ such that $\underline{{\rm div}}_\nu(\underline{X}) \in L^2(M, \nu)$.

\begin{theorem}\label{LucicPasresult} Let $1 < p < \infty$. There exists a unique surjective $L^\infty(\nu)$-linear and continuous operator $$P_\nu : \Gamma_p(T^*M;\nu) \rightarrow L_\nu^p(T^*M)$$
such that $P_\nu(\underline{d} f) = df$ for every $f \in C_c^1(M)$. Moreover,
$$\vert P_\nu(\underline {\omega}) \vert \leq \vert \underline {\omega} \vert, \quad \mbox{for every} \ \underline {\omega} \in \Gamma_p(T^*M;\nu),$$
and
\begin{equation}\label{ookk}
\hbox{for any $\omega \in L_\nu^p(T^*M)$ there exists \  $\underline {\omega} \in P_\nu^{-1}(\omega)$ \ such that \ $\vert \omega \vert = \vert \underline {\omega} \vert$ $\nu$-a.e.}
\end{equation}

Let us denote by $$i_\nu : L^q_\nu(T M) \rightarrow \Gamma_q(TM;\nu)$$
the adjoint map of $P_\nu$, i.e., the unique $L^\infty(\nu)$-linear and continuous operator satisfying
$$ \underline {\omega} (i_\nu(v)) = P_\nu( \underline {\omega})(v) \quad \hbox{$\nu$-a.e.} \quad \hbox{for every} \ v \in L^q_\nu(T M) \ \hbox{and} \ \underline {\omega} \in  \Gamma_q(TM;\nu).$$
Then, it holds that
$$\vert i_\nu(v) \vert = \vert v \vert \quad \hbox{$\nu$-a.e.} \quad \hbox{for every} \ v \in L^q_\nu(T M).$$
\end{theorem}

Now we are going to characterize the $p$-Laplacian operator in the  metric measure space $(M, d_F, \nu)$ associated to a geodesically complete, reversible Finsler manifold $(M,F)$ and a non-negative Radon measure $\nu$ on $M$.

\begin{theorem}\label{CharactFinsler}
Let $1 < p < \infty$ and let $\Delta_{p,\nu}$ be the $p$-Laplacian in $(M, d_F, \nu)$. For $v \in L^2(M, \nu)$, we have $-\Delta_{p,\nu} u = v$ if and only if $u \in W^{1,p}(M, d_F, \nu)$ and there exists a vector field $\underline{X} \in D^{q,2}(\underline{{\rm div}}_\nu)$ satisfying the following condition:
\begin{equation}\label{E1ChatN}
\vert \underline{X}  \vert^q  \leq \vert Du \vert^p_\nu \quad \hbox{$\nu$-a.e. in} \ M,
\end{equation}
\begin{equation}\label{E2ChatN}
- \underline{{\rm div}}_\nu(\underline{X}) = v\quad \hbox{in} \ M,
\end{equation}
\begin{equation}\label{E3ChatN}
\hbox{there exists} \ \underline{u} \in P_\nu^{-1}(du) \in  \Gamma_q(TM;\nu), \ \hbox{such that} \ \underline{u}(\underline{X}) = \vert \underline{u} \vert^p \ \hbox{$\nu$-a.e. \ in} \ M.
\end{equation}
\end{theorem}
\begin{proof} Suppose that $-\Delta_{p,\nu} u = v$. Then, $u \in W^{1,p}(M, d_F, \nu)$ and there exists a vector field $X \in \mathcal{D}^{q,2}(M)$ with  $| X |^q \leq | du |_*^p$ $\nu$-a.e. such that
\begin{equation}\label{E01ChatN} -\mbox{div}(X) = v, \quad  \hbox{in $M$},
\end{equation}
\begin{equation}\label{E02ChatN}
du(X) = |du|_*^p  \quad \nu\hbox{-a.e. in } M.
\end{equation}
Since $X \in L^q_\nu(T M)$, we have  $\underline{X}:= i_\nu(X) \in \Gamma_q(TM;\nu)$. Then,
$$\vert \underline{X}  \vert^q = \vert  i_\nu(X) \vert^q = \vert X \vert^q \leq | du |_*^p =\vert Du \vert^p_\nu \quad \hbox{$\nu$-a.e. in} \ M, $$
and \eqref{E1ChatN} holds.

For $f\in C_c^\infty(M)$, by \eqref{E01ChatN}, we have
$$\int_{M} f  v \, d \nu  = - \int_{M}f \mbox{div}(X) \, d\nu =  \int_{M} df(X) \, d\nu = \int_{M} \underline{d}f(\underline{X}) \, d\nu = - \int_{M} f \, \underline{{\rm div}}_\nu(\underline{X}) \, d \nu, $$
and \eqref{E2ChatN} holds.

By \eqref{ookk}, there exists $\underline{u} \in P_\nu^{-1}(du)$ such that $\vert \underline{u} \vert = \vert du \vert_*$. Then, by \eqref{E02ChatN}, we have for $\nu$-a.e. in $M$,
$$\vert \underline{u} \vert^p =  \vert du \vert_*^p = du(X) = P_\nu(\underline{u})(X) = \underline{u}(i_\nu (X)) = \underline{u}(\underline{X}),$$
and \eqref{E3ChatN} holds.

Reciprocally, suppose that $u \in W^{1,p}(M, d_F, \nu)$ and there exists a concrete vector field  $\underline{X} \in  D^{q,2}(\underline{{\rm div}}_\nu)$ satisfying \eqref{E1ChatN}, \eqref{E2ChatN} and \eqref{E3ChatN}. Let $X \in i_\nu^{-1}(\underline{X}) \in L^q_\nu(TM)$, then
$$\vert X \vert^q = \vert  i_\nu(X) \vert^q = \vert \underline{X}  \vert^q \leq \vert Du \vert^p_\nu = | du |_*^p.$$
Given $f\in C_c^\infty(\R^N)$, by \eqref{E2ChatN}, we have
$$\int_{M} v f \, d \nu = - \int_{M} f \, \underline{{\rm div}}_\nu(\underline{X}) \, d \nu = \int_{M} \underline{d}f(\underline{X}) \, d \nu =  \int_M df(X) \, d\nu =  - \int_{M} f \, \mbox{div}(X) \, d\nu.$$
Now, since $C_c^\infty(\R^N)$ is dense in $W^{1,p}(M, d_F, \nu)$, for any $g \in W^{1,p}(M, d_F, \nu)$
$$\int_{M} v g \, d \nu = - \int_{M} g \, {\rm div} (X) \, d\nu = \int_{M} dg(X) \, d \nu,$$
and therefore \eqref{E01ChatN} holds.

Finally, by \eqref{E3ChatN}, we have for $\nu$-a.e. in $\R^N$,
$$\vert du \vert_*^p = \vert \underline{u} \vert^p = P_\nu(\underline{u})(X) = du(X),$$
and \eqref{E02ChatN} holds.
\end{proof}

{ Again, using this characterisation of the operator $\Delta_{p,\nu}$ one can immediately obtain a corresponding characterisation of solutions to the Cauchy problem for the $p$-Laplacian in a Finsler manifold $(M, d_{F},\nu)$ in terms of concrete vector fields.}

\begin{remark}{\rm For $p\not=2$, the above result is new. For $p=2$ (i.e. for the heat flow), to the best of our knowledge the only results in Finsler manifolds are the ones obtained by Ohta and Sturm \cite{OS}, and the results in \cite{AIS} for the particular case $(\R^N, H)$ with $H$ a norm in $\R^N$. When $p\not=2$, to the best of our knowledge the only case that was studied is the elliptic problem for the particular case $(\R^N, H)$ with $H$ a norm in $\R^N$, see for instance \cite{BKJ} and the references therein. $\blacksquare$}
\end{remark}

\bigskip

\noindent {\bf Acknowledgment.} The first author has been partially supported by the DFG-FWF project FR 4083/3-1/I4354, by the OeAD-WTZ project CZ 01/2021, and by the project 2017/27/N/ST1/02418 funded by the National Science Centre, Poland. The second author has been partially supported by the Spanish MCIU and FEDER, project PGC2018-094775-B-100.

\end{document}